\documentclass[a4paper,leqno,11pt]{article}
\usepackage{latexsym}
\usepackage{amsmath}
\usepackage{amssymb}
\usepackage{enumerate}

\usepackage{theorem}
\newtheorem{theorem}{Theorem}[section]
\newtheorem{proposition}[theorem]{Proposition}
\newtheorem{lemma}[theorem]{Lemma}

\newtheorem{condition}[theorem]{C}

\theorembodyfont{\rmfamily}
\newtheorem{proof}{\textmd{\textit{Proof.}}}

\newtheorem{remark}[theorem]{Remark}

\newtheorem{definition}[theorem]{Definition}

\makeatletter

\@addtoreset{equation}{section}
\makeatother

\newcommand{\qedd}{\hfill \Box}

\newcommand{\R}{\ensuremath{\mathbb{R}}}
\newcommand{\Sph}{\ensuremath{\mathbb{S}}}
\newcommand{\bbC}{\ensuremath{\mathbb{C}}}
\newcommand{\E}{\ensuremath{\mathbb{E}}}

\title{Generalized Finsler structures on closed 3-manifolds
\footnote{
Mathematics Subject Classification (2000)\,:\,
 Primary 53C60; 
 Secondary 53D35
.}
}
\author{Sorin V. Sabau \footnote{ This research was partly supported by Grant-in-Aid for Scientific Research (C) (No. 22540097), Japan Society for the Promotion of Science.}, Kazuhiro Shibuya, 
Gheorghe Piti\c s}
\date{\textit{In memory of Professor Dr. Stere Ianu\c s}}
\begin{document}


\maketitle

\begin{abstract}
An $(I,J,K)$-generalized Finsler structure on a 3-manifold is a generalization
of a Finslerian structure, introduced in order to separate and clarify 
 the local and global aspects in Finsler geometry making use of the Cartan's method of exterior differential systems (\cite{Br2002}). 

In this paper, we show that there is a close relation between $(I,J,1)$-generalized Finsler structures and a class of contact circles, namely the so-called {\it Cartan structures} (\cite{GG1995}).

This correspondence allows us to determine the topology of 3-manifolds that admit $(I,J,1)$-generalized Finsler structures and to single out classes of $(I,J,1)$-generalized Finsler structures induced by standard Cartan structures.
\end{abstract}


\section{Introduction}\label{sec: intro}

\ \ \  A classical Finsler structure $(M,F)$ is a smooth manifold $M$ endowed with a Banach norm on each tangent space $T_{x}M$ that varies smoothly with the base point all over the manifold, for any $x\in M$. A Riemannian manifold is a particular case when each of these Banach norms are induced by a quadratic form. Geometrically, this is equivalent to the choice of a unit sphere in each tangent space, such that one obtains a smooth hypersurface $\Sigma\subset TM$ which has the property that each fiber $\Sigma_{x}:=\Sigma\cap T_{x}M$ is a smooth, strictly convex hypersurface in $TM$ which surrounds the origin $O_{x}\in T_{x}M$. 

Except the preference for local computations, a peculiarity of Finsler structures is that, unlike the Riemannian case, one has no means to specify a canonical Finsler structure on a given manifold, therefore, constructing models for Finslerian structures with given geometrical properties (such as constant flag curvature) is an important topic that rises interesting questions about the local and global generality of such structures. 

A generalization of classical Finsler structures has been introduced by R. Bryant by defining the notion of $(I,J,K)$-generalized Finsler structures (see \cite{Br1996}), namely a 3-manifold $\Sigma$ endowed with a coframing satisfying some specific structure equations (see Section \ref{sec: FinslerStructures} for the precise definition). We use here only such structures on 3-manifolds, but these can be defined in any dimension (see \cite{Br2002}). Generalized Finsler structures were introduced with the specific intention of \lq micro-localization' of classical Finsler structures that allows separating the local geometrical properties of coframes satisfying certain differential geometric conditions, or solving PDE's, from the global geometrical properties of the manifolds $\Sigma$ or $M$ related with the behavior of the leaf space of certain foliations. 

There are a lot of questions and problems that this new notion brings about. For instance, the absence of results on the existence of global defined Finsler structures motivates one to study the existence of global defined generalized Finsler structures on a 3-manifold $\Sigma$ as well as the case when this is realizable as a classical Finsler structure on a surface $M$. For the case of constant flag curvature one, the only available constructions are Bryant's. In particular, making use of generalized Finsler structures, he was able to construct for the first time global defined Finsler structures, of constant flag curvature one, on spheres (\cite{Br1996}, \cite{Br2002}), proving in this way the importance of generalized Finsler structures and that a more detailed study of these is worthy.

The existence of  Finsler structures on surfaces with vanishing Landsberg scalar, i.e.  $J=0$ (see Section \ref{sec: FinslerStructures} for definition), is an old open problem in Finsler geometry (see \cite{SSS2010} for the background of the problem). A progress in solving this problem was obtained by showing the existence of non-trivial generalized Finsler structures with $J=0$ (\cite{SSS2010}, \cite{SSS2011}) proving one more time the incontestable utility of generalized Finsler structures. 

Therefore, since the essential ingredient used by Bryant in constructing generalized Finsler structures is the contact structure, it is natural to attempt the use of refined contact topology methods, that have been developed within the last 25 years. 

On the other hand, let us also recall a classical result, namely that any closed, oriented 3-manifold admits a parallelization by contact forms (see for example \cite{GG1997}, \cite{G1987}). It worth mentioning that the history behind this results has started with the following S. S. Chern's simple, but extremely fruitful question in 1966: \lq\lq Does a closed, oriented 3-manifold always admit contact structures?'' The answer is affirmative and it was given in 1971 by R. Lutz proving the existence on the 3-sphere, and by J. Martinet in the general case. 

Taking these into account, it is natural to ask the problem of existence of two or three linearly independent contact forms on 3-manifolds satisfying supplementary conditions, such as to determine the same volume, or even more, that any $\Sph^{1}$-, or $\Sph^{2}$-, linear combination to determine the same volume, respectively. In this way one obtains the notions of taut contact circle and taut contact sphere, respectively, introduced and studied by H. Geiges and J. Gonzalo in \cite{GG1995}, \cite{GG1997}, \cite{GG2002}. These notions turn out to be extremely fruitful leading to a complete classification of these structures using the 8-geometries of Thurston, moduli space dimension and many other interesting results (see especially \cite{GG1995}, \cite{GG2002}).

We believe that all these are strong enough reasons to motivate our attempt hereafter to apply H. Geiges and J. Gonzalo's results and methods in the study of global defined $(I,J,K)$-generalized Finsler structures on 3-manifolds.

\bigskip

In the present paper we study the relation between $(I,J,K)$-generalized Finsler structures on closed 3-manifolds and taut contact circles defined on the same manifold, in particular Cartan structures (see \cite{GG1995}, \cite{GG2002} for details on these structures). 

Our study goes both directions. We show that an $(I,J,K)$-generalized Finsler structure, defined on a closed 3-manifold $\Sigma$, naturally induces a taut contact circle on $\Sigma$, which is in fact a $\mathcal K$-Cartan structure, provided $K=1$ (Proposition \ref{prop: correspondence}). Conversely, if we start with a $\mathcal K$-Cartan structure, then the coframe \eqref{coframe induced IJ1GFS} is an $(I,J,1)$-generalized Finsler structure on a quotient manifold $\Sigma=G\slash\Gamma$, provided we are able to find a $\Gamma$-invariant 1-form $\varphi$ on $G$ that satisfies the structure equation \eqref{phi_struct_eq} with non-constant coefficients. Here $G$ and $\Gamma$ have the meanings in the theorem below. This approach allows us to obtain several new 
$(I,J,1)$-generalized Finsler structures on $\Sigma$ and to write explicitly their form (Section \ref{sec:4}).

The $(I,J,K)$-generalized Finsler structures are more general geometrical structures than taut contact circles and Cartan structures, however there is very few that we know about them. The present paper shows how the topology of closed 3-manifolds $\Sigma$ that admit an $(I,J,1)$-generalized Finsler structure is restricted. Indeed, here is our main result:

\begin{theorem}\label{thm: GFS classif}
Let $\Sigma$ be a closed 3-manifold. Then $\Sigma$ admits an $(I,J,1)$-generalized Finsler structure if and only if 
it is diffeomorphic to a quotient of the Lie group $G$ 
under a discrete subgroup $\Gamma$ of $G$, where $G$ is one of the following:
\begin{enumerate}
\item $\Sph^3=SU(2)$, the universal cover of $SO(3)$,
\item $\widetilde{SL}_2$, the universal cover of $PSL_2(\mathbb R)$,
\item $\widetilde{E}_2$, the universal cover of the Euclidean group, i.e. orientation preserving isometries of $\mathbb R^2$.
\end{enumerate}
\end{theorem}

In the Theorem above both trivial and non-trivial cases of $(I,J,1)$-generalized Finsler structures are included. However, we show that for each $G\in\{SU(2),\widetilde E_{2},\widetilde{SL}_{2}\}$ there exist non-trivial globally defined $(I,J,1)$-generalized Finsler structures on each 3-manifold of type $G\slash\Gamma$ for a discrete subgroup $\Gamma$ of $G$.

Concrete constructions and local forms are given in Section \ref{sec:4}. Indeed, the easier way to obtain such $(I,J,1)$-generalized Finsler structures is to start with Liouville-Cartan structures obtained from Riemannian surfaces and then apply the construction we gave in Section \ref{sec:5}. Another method is to work directly on $G$ with $\Gamma$-invariant 1-forms $\varphi$ as we do in Section \ref{sec:6.1}. 
 
 In Section \ref{sec:6} we study $(I,J,1)$-generalized Finsler structures induced by Cartan structures in the context of conformal classes of taut contact circles on a closed 3-manifold $\Sigma$. In the case of $SU(2)\slash\Gamma$ we point out a 1-to-1 correspondence of a special class of $(I,J,1)$-generalized Finsler structures, namely the $\mathcal K$-induced $(I,J,1)$-generalized Finsler structures (see Section \ref{sec:6} for definition), and $\mathcal K$-Cartan structures with $\mathcal K>0$, fact that allows to compute the moduli space dimension of these $(I,J,1)$-generalized Finsler structures in a special case. 
 
 We  remark that the construction of taut contact structures from $(I,J,K)$-generalized structures (in Proposition \ref{prop: correspondence}) is far from being the only one. We describe in Appendix another way of doing this, namely we pair $(I, 0, K)$-generalized Finsler structures with $\mathcal K$-Cartan structures linking in this way the present paper with our past work \cite{SSS2010} and \cite{SSS2011}. In this case the topology of $\Sigma$ is also restricted in a similar way as described in Theorem \ref{thm: GFS classif}.

Our present study is in the same time a generalization of the work of Bryant's (\cite{Br1996}, \cite{Br2002}) where $(I,J,1)$-generalized Finsler structures are constructed by means of a Zoll metric $(\Sph^{2},g)$ which gives a classical Finsler structure on the round sphere $\Sph^{2}$ (compare with our constructions in Section \ref{sec:6.1}). Interpreted in the context of taut contact circles our study gives a more geometrical explanation of the constructions in \cite{Br2002}. 

Finally, we point out that the present paper is only the beginning of the study of $(I,J,K)$-generalized Finsler structures on quotient manifolds $G\slash\Gamma$ and there are many things left to be clarified in the future. For a given $(I,J,K)$-generalized Finsler structure on a closed 3-manifold $\Sigma$, we are mainly interested in relating the discrete subgroups $\Gamma$ of $G$ with the Finslerian isometry groups acting on the space $M:=\Sigma\slash\{\omega^{1},\omega^{2}\}$ that is an orbifold in the general case. We have clarified in the present paper the relation between taut contact circles, more precisely $\mathcal K$-Cartan structures, and $(I,J,1)$-generalized Finsler structures on $\Sigma$, but relating other types of contact circles (see \cite{GG1997} for definitions) with $(I,J,K)$-generalized Finsler structures is still an open problem. These topics, as well as many others,  might be the subject of a forthcoming paper.

\bigskip

{\bf Acknowledgments.}
We would like to express our thanks to Prof. Hideo Shimada and Martin Guest who supported us with many useful suggestions during the preparation of this manuscript. We are also indebted to Prof. Reiko Miyaoka for suggestions that have improved the manuscript.


\section{Finsler and Generalized Finsler Structures}\label{sec: FinslerStructures}


 \quad Let  us start by recalling that a Finsler norm on a real smooth, $n$-dimensional manifold
$M$ is a function $F:TM\to \left[0,\infty \right)$ that is positive and
smooth on $\widetilde{TM}=TM\backslash\{0\}$, has the {\it homogeneity property}
$F(x,\lambda v)=\lambda F(x,v)$, for all $\lambda > 0$ and all 
$v\in T_xM$, having also the {\it strong convexity} property that the
Hessian matrix
\begin{equation}\label{Hessian matrix}
g_{ij}=\frac{1}{2}\frac{\partial^2 F^2}{\partial y^i\partial y^j}(x,y)
\end{equation}
is positive definite at any point $u=(x^i,y^i)\in \widetilde{TM}$.

 The fundamental function $F$ of a Finsler structure $(M,F)$ determines and it is determined by the (tangent) {\it indicatrix}, or the total space of the unit tangent bundle 
$\Sigma_F:=\{u\in TM:F(u)=1\}$,
which is a smooth hypersurface of $TM$ such that at each $x\in M$ the {\it indicatrix at x}
$\Sigma_x:=\{v\in T_xM \ |\  F(x,v)=1\}=\Sigma_F\cap T_xM$
is a smooth, closed, strictly convex hypersurface in
$T_xM$. 
   
   A Finsler structure $(M,F)$ can be therefore regarded as smooth
hypersurface $\Sigma\subset TM$ for which the canonical projection
$\pi:\Sigma\to M$ is a surjective submersion and having the property
that for each $x\in M$, the $\pi$-fiber $\Sigma_x=\pi^{-1}(x)$ is
strictly convex including the origin $O_x\in T_xM$. 

A generalization of this notion is the generalized Finsler structure introduced by R. Bryant (see \cite{Br1996}, \cite{Br2002} for definitions and fundamental properties, as well as \cite{SSS2010} and \cite{SSS2011} for some recent developments). 

\begin{definition}\label{GFS}
A 3-dimensional manifold $\Sigma$ endowed with a coframing $
\omega=(\omega^1,\omega^2,\omega^3)$ which satisfies the structure equations 
\begin{equation}\label{finsler_struct_eq}
\begin{split}
d\omega^1&=-I\omega^1\wedge\omega^3+\omega^2\wedge\omega^3\\
d\omega^2&=-\omega^1\wedge\omega^3\\
d\omega^3&=K\omega^1\wedge\omega^2-J\omega^1\wedge \omega^3
\end{split}
\end{equation}
will be called an $(I,J,K)$-{\it generalized Finsler structure} on $\Sigma$,
where $I$, $J$, $K$ are smooth functions on $\Sigma$, called {\it the invariants} of the generalized Finsler structure $
(\Sigma,\omega)$.
\end{definition}


As pointed out in  \cite{Br1996}, the difference between a classical Finsler structure and a 
generalized one is global in nature, in the sense that {\it every generalized Finsler structure on a 3-manifold is locally 
diffeomorphic to a classical Finsler surface structure. }

By taking the exterior derivative of the structure equations (\ref{finsler_struct_eq}) one obtains the {\it Bianchi equations} 
\begin{equation}\label{GFS_Bianchi}
   J=I_2,\quad  K_3+KI+J_2=0,
\end{equation}
 where we denote by subscripts the directional derivatives with respect to the coframing $\omega$, i.e.
$df=f_1\omega^1+f_2\omega^2+f_3\omega^3$,
for any smooth function $f$ on $\Sigma$.
	
	Taking now one more exterior derivative of the last formula written above,  one obtains the {\it Ricci identities} 
with respect to the generalized Finsler structure
\begin{equation*}
\begin{split}
& f_{21}-f_{12}=-Kf_3\\
& f_{32}-f_{23}=-f_1\\
& f_{31}-f_{13}=If_1+f_2+Jf_3.
\end{split}
\end{equation*}

 As long as we work only with generalized Finsler surfaces, it might be possible that this generalized structure 
does not lead to a classical Finsler structure on a surface $M$. Indeed,
the following fundamental result gives necessary and sufficient conditions for a generalized Finsler structure to be a Finsler structure (\cite{Br1996}):

\begin{theorem}\label{thm: GFS equivalence to Finsler}
The necessary and sufficient conditions for an $(I,J,K)$-generalized Finsler structure  $(\Sigma,\omega)$ to be realizable as a classical Finsler structure on a surface $M$ are
\begin{enumerate} 
\item the leaves of the codimension two foliation $\mathcal F=\{\omega^1=0,\ \omega^2=0\}$ are compact;
\item it is amenable, i.e. the leaf space $M$ of the foliation $\mathcal F$ is a smooth surface 
such that the natural projection $\pi:\Sigma\to M$   is a smooth submersion;
\item the canonical immersion $\iota:\Sigma\to TM$, given by 
$\iota(u)=\pi_{*,u}(\hat{e}_2)$, is one-to-one on each $\pi$-fiber $\Sigma_x$, 
where $(\hat e_{1},\hat e_{2},\hat e_{3})$ is the dual frame of the coframing $\omega=(\omega^{1}, \omega^{2}, \omega^{3})$.
\end{enumerate}
 \end{theorem}
 
 \begin{remark}
 We point out that the definitions of generalized Finsler structures given in \cite{Br1996} and \cite{Br2002} are slightly different. Our Definition \ref{GFS} is the same as \cite{Br1996}, while the definition in \cite{Br2002} adds supplementary conditions (see Definition 1 in \cite{Br2002} for details). 
 \end{remark}

Special cases of generalized Finsler structures are easily obtained by taking particular values for the structure functions $I$, $J$, $K$.

If $I=0$, then from Bianchi equations one obtains $J=0$ and the resulting generalized Finsler structures will be called {\it trivial} hereafter. A $(0,0,K)$-generalized Finsler structure is a $K$-Cartan structure. We will discuss this type of structure in detail in next section. A  
$(0,0,K)$-generalized Finsler structure that satisfies the conditions in Theorem \ref{thm: GFS equivalence to Finsler} above induces a Riemannian metric of Gauss curvature $K$ on the leaf space $M$ of the foliation $\{\omega^1=0,\ \omega^2=0\}$.

Another special case is an $(I,J,1)$-generalized Finsler structure. Such a structure satisfying the conditions in Theorem \ref{thm: GFS equivalence to Finsler} provides a Finsler structure of constant flag curvature on the leaf space $M$ (see \cite{BCS2000} and \cite{Br2002} for details on constant flag curvature Riemann-Finsler structures).

 Let us point out, for later reference, that an $(I,J,1)$-generalized Finsler structure is a coframe 
 $(\omega^1,\omega^2,\omega^3)$ on the 3-manifold $\Sigma$ that verifies the structure equations 
\begin{equation}\label{eq: finslerK=1_struct_eq}
\begin{split}
d\omega^1&=-I\omega^1\wedge\omega^3+\omega^2\wedge\omega^3\\
d\omega^2&=-\omega^1\wedge\omega^3\\
d\omega^3&=\omega^1\wedge\omega^2-J\omega^1\wedge \omega^3.
\end{split}
\end{equation}

In this case, the Bianchi equations read
\begin{equation}\label{eq: K=1_Bianchi}
 J=I_2, \qquad I=-J_2.
\end{equation}
\section{Contact circles}\label{sec: Contact circles}

A quick look at the structure equations \eqref{finsler_struct_eq} shows that for a generalized Finsler structure 
$(\omega^1,\omega^2,\omega^3)$ on a 3-manifold $\Sigma$, the 1-forms $\omega^1$ and $\omega^2$ are contact forms. In the \lq\lq non-flat" case $K\neq 0$ everywhere on $\Sigma$, one can easily see that $\omega^3$ is also a contact form. 

All these suggest that some concepts and ideas in contact geometry might be useful in the study of generalized Finsler structures.

Let us recall here few basic facts (for details see \cite{GG1995}, \cite{GG2002}).

%


\begin{definition}\label{def: contact circle}
A 3-manifold $\Sigma$ is said to admit a {\it contact circle} if it admits a pair of contact forms $(\alpha^1,\alpha^2)$ such that for any $(\lambda_1,\lambda_2)\in S^1$, i.e. $\lambda_1,\lambda_2\in \mathbb R$, $(\lambda_1)^2+(\lambda_2)^2=1$, the linear combination $\lambda_1\alpha^1+\lambda_2\alpha^2$ is also a contact form.
\end{definition}


Concerning contact circles it is known
that on every closed, orientable 3-manifold there are contact circles realizing any of the two orientations (\cite{GG1997}, Theorem 1.2).
%
 
 \begin{definition}\label{def: taut contact}
A contact circle $(\alpha^1,\alpha^2)$ is called a  {\it taut contact circle} if the contact forms  $\lambda_1\alpha^1+\lambda_2\alpha^2$ define the same volume form for all $(\lambda_1,\lambda_2)\in S^1$.

\end{definition}

By straightforward computation one can verify  that 
the contact circle $(\alpha^1,\alpha^2)$ is a taut contact circle if and only if the conditions
\begin{equation}
\begin{split}
& \alpha^1\wedge d\alpha^1=\alpha^2\wedge d\alpha^2\neq 0\\
& \alpha^1\wedge d\alpha^2+\alpha^2\wedge d\alpha^1=0
\end{split}
\end{equation}
hold good.

The following definition  (\cite{GG1995}) is also natural. 

\begin{definition}  
The contact circle $(\alpha^1,\alpha^2)$ is called a {\it Cartan structure} on the 3-manifold $\Sigma$ if the following conditions are satisfied
\begin{equation}
\begin{split}
& \alpha^1\wedge d\alpha^1=\alpha^2\wedge d\alpha^2\neq 0\\
& \alpha^1\wedge d\alpha^2=0, \quad \alpha^2\wedge d\alpha^1=0.
\end{split}
\end{equation}
\end{definition}

%
%
%

It results immediately

\begin{lemma}{\rm(\cite{GG2002})}\label{Lemma 3.3}

If $(\alpha^1,\alpha^2)$ is a Cartan structure on the 3-manifold $\Sigma$, then there exists a unique 1-form $\eta$ on $\Sigma$ such that
\begin{equation}\label{Cartan_structure}
 d\alpha^1=\alpha^2\wedge\eta, \quad
 d\alpha^2=\eta\wedge\alpha^1.
\end{equation}
\end{lemma}

From this Lemma it follows that $(\alpha^1,\alpha^2,\eta)$ is a coframe on $\Sigma$. Indeed, taking into account that $\alpha^1$, $\alpha^2$ are contact forms and 
\eqref{Cartan_structure} it follows $\alpha^1\wedge\alpha^2\wedge\eta\neq 0$.

The structure equations of the 1-form $\eta$ can be complicated in an arbitrary point of $\Sigma$, without further conditions. A special case of Cartan structure is given in the following definition.


\begin{definition} (\cite{GG2002})
The Cartan structure $(\alpha^1,\alpha^2)$ is called a {\it $\mathcal K$-Cartan structure} if the unique form $\eta$ in Lemma \ref{Lemma 3.3} satisfies the structure equation
\begin{equation}
d\eta=\mathcal K \alpha^1\wedge\alpha^2.
\end{equation}
\end{definition}

By abuse of language the coframe $(\alpha^1,\alpha^2,\eta)$ on the 3-manifold $\Sigma$ is called a $\mathcal K$-Cartan structure if it satisfies the structure equations
\begin{equation}\label{K_Cartan_structure}
\begin{split}
& d\alpha^1=\alpha^2\wedge\eta\\
& d\alpha^2=\eta\wedge\alpha^1\\
& d\eta=\mathcal K \alpha^1\wedge\alpha^2.
\end{split}
\end{equation}

Obviously, a $\mathcal K$-Cartan structure coincides with a $(0,0,\mathcal K)$-generalized Finsler structure.

One can easily remark that for any $\mathcal K$-Cartan structure, the differential of the structure function $\mathcal K$ must satisfy
$d\mathcal K=\mathcal K_1\alpha^1+\mathcal K_2\alpha^2$,
i.e. the function $\mathcal K$ lives on the leaf space of the codimension two foliation $\{\alpha^1=0,\alpha^2=0\}$. In general this is not necessarily a differentiable manifold. 

One of the main results of this theory is the following 


\begin{theorem} {\rm (\cite{GG1995})}\label{thm: classification}

Let $\Sigma$ be a closed 3-manifold. Then $\Sigma$ admits a taut contact circle if and only if 
$\Sigma$ is diffeomorphic to a quotient of the Lie group $G$ 
under a discrete subgroup $\Gamma$ of $G$, acting by left multiplication, where $G$ is one of the following:
\begin{enumerate}
\item $\Sph^3=SU(2)$, the universal cover of $SO(3)$,
\item $\widetilde{SL}_2$, the universal cover of $PSL_2(\mathbb R)$,
\item $\widetilde{E}_2$, the universal cover of the Euclidean group, i.e. orientation preserving isometries of $\mathbb R^2$.
\end{enumerate}
\end{theorem}


\subsection{Standard Cartan structures}
If one denotes by $\alpha^0$ the Maurer-Cartan form on the Lie group $G$, where $G$ is one of the Lie groups in Theorem \ref{thm: classification}, then we can write
\begin{equation}
\alpha^0=\alpha^1 e_1+\alpha^2 e_2+\alpha^3 e_3,
\end{equation}
where $(e_1,e_2,e_3)$ is a basis for the Lie algebra $\mathfrak g$ of the Lie group $G$, 
and we can assume that this basis satisfies the following structure equations
\begin{equation}
\begin{split}
& [e_1,e_2]=-\varepsilon e_3\\
& [e_2,e_3]= -e_1\\
& [e_3,e_1]= -e_2,
\end{split}
\end{equation}
where $\varepsilon=1$ for $SU(2)$, $\varepsilon=-1$ for $\widetilde{SL}_2$, and $\varepsilon=0$ for $\widetilde{E}_2$.  

Equivalently, we obtain the structure equations
\begin{equation}\label{str_eq_for_stand_Cartan}
\begin{split}
& d\alpha^1=\alpha^2\wedge\alpha^3\\
& d\alpha^2=\alpha^3\wedge\alpha^1\\
& d\alpha^3=\varepsilon\alpha^1\wedge\alpha^2,
\end{split}
\end{equation}
where $(\alpha^1,\alpha^2,\alpha^3)$ is the dual coframe of $(e_1,e_2,e_3)$, and $\varepsilon$ has the same meaning as above. These structures are $1$, $0$, $-1$-Cartan structures on $SU(2)$, $\widetilde{E}_2$ and $\widetilde{SL}_2$, respectively. They are called {\it standard Cartan structures}.

For any given discrete subgroup $\Gamma$ of $G$, it is obvious now that $(\alpha^1,\alpha^2)$ is a Cartan structure on $G$ which descends to any left-quotient 
$G\slash\Gamma$.


\subsection{Liouville-Cartan structures}

For a Riemannian surface $(\Lambda,g)$ with local coordinates $(x^1,x^2)$ let us consider its cotangent bundle 
$T^*\Lambda$ with local coordinates $(x^1,x^2,p_1,p_2)$. Then 
\begin{equation}
\begin{split}
& \theta^1=p_1dx^1+p_2dx^2\\
& \theta^2=-p_2dx^1+p_1dx^2.
\end{split}
\end{equation}
is a pair of 1-forms on $T^*\Lambda$.

If we denote now the unit cotangent bundle of $(\Lambda,g)$ by $ST^{*}\Lambda$, then these forms restricted to $ST^*\Sigma$ are called the {\it Liouville-Cartan 1-forms} associated to the Riemannian surface $(\Lambda,g)$.

One can easily see that these forms are in fact the tautological 1-forms on the Riemannian surface $(\Lambda,g)$.

We will see later (Section \ref{sec:5}) how are Liouville-Cartan structures constructed on the quotient manifolds in Theorem \ref{thm: classification}.

\section{The correspondence between Generalized Finsler and Cartan structures}

Let us firstly remark that the structure equations of an $(I,J,K)$-generalized Finsler structure imply
 $(\omega^i\wedge d\omega^j)=A\Omega$, for $i,j\in\{1,2,3\}$, where the structure matrix $A=(a^{ij})$ is given by
\begin{equation}
 A=\begin{pmatrix}
  1 & 0 & 0 \\
  I & 1 & J \\
  0 & 0 & K 
 \end{pmatrix}
\end{equation}
and $\Omega:=\omega^1\wedge\omega^2\wedge\omega^3\neq 0$ is the volume form.
From here we obtain immediately:

%
\begin{proposition}\label{prop: correspondence}
Let $(\Sigma, \omega)$
be an $(I,J,K)$-generalized Finsler structure on a closed 3-manifold 
$\Sigma$, where $\omega=(\omega^1,\omega^2,\omega^3)$.
Then we have
\begin{enumerate}
\item $(\omega^1,\omega^2)$ is a taut contact circle if and only if $I=0$, i.e. $(\Sigma,\omega)$ is in fact a $\mathcal K:=K$-Cartan structure;
\item $(\omega^1,\omega^3)$ is a taut contact circle if and only if $K=1$, i.e. $(\Sigma,\omega)$ is 
an $(I,J,1)$-generalized Finsler structure. This taut contact circle is actually a $\mathcal K$-Cartan structure on $\Sigma$;
\item $(\omega^2,\omega^3)$ is a taut contact circle if and only if $K=1$ and $J=0$. Moreover, $I=0$ and $(\Sigma,\omega)$ is a $1$-Cartan structure.
\end{enumerate}
\end{proposition}

The third conclusion follows from the Bianchi equations \eqref{GFS_Bianchi}.

\begin{remark}\label{rem: from GFS to Cartan}
It is straightforward from the second statement in Proposition \ref{prop: correspondence} 
that an $(I,J,1)$-generalized Finsler structure on a 3-manifold $\Sigma$ naturally induces a $\mathcal K$-Cartan structure. Indeed, 
if 
$\Sigma$ admits an $(I,J,1)$-generalized Finsler structure $(\omega^1,\omega^2,\omega^3)$, then the pair of 1-forms $\alpha^1:=\omega^1$, $\alpha^2:=\omega^3$ is a Cartan structure, i.e. a taut contact circle on $\Sigma$ and the conclusion follows from Theorem \ref{thm: classification}. A simple computation shows that there exists a 1-form
\begin{equation}\label{eta}
\alpha^3=I\omega^1+J\omega^3-\omega^2
\end{equation}
such that $(\alpha^1,\alpha^2,\alpha^3)$ is a $\mathcal K$-Cartan structure on $\Sigma$ with the structure function 
\begin{equation}\label{formula for K}
\mathcal K=-I^2 -J^2 +J_1 -I_3+1.
\end{equation} 
Here all the subscripts are with respect to the coframe $\omega$.
\end{remark}

The converse is not so obvious. We are interested in finding if there exist non-trivial $(I,J,1)$-generalized Finsler structures, i.e. structures having the functions $I$ and $J$ not constant on $\Sigma$. We will show in the following sections that the standard $\mathcal K$-Cartan structures induce non-trivial  $(I,J,1)$-generalized Finsler structures.

\begin{proposition}\label{prop4.3}
If $\alpha=(\alpha^1,\alpha^2,\alpha^3)$ is a $\mathcal K$-Cartan structure on $\Sigma$, then the induced coframe
\begin{equation}\label{coframe in I, J}
\begin{split}
\omega^1&= \alpha^1\\
\omega^2&=I\alpha^1+J\alpha^2-\alpha^3\\
\omega^3&= \alpha^2
\end{split}
\end{equation}
is an $(I,J,1)$-generalized Finsler structure on $\Sigma$ if and only if the functions $I,J:\Sigma\to\R$ are solutions of the following directional PDE system 
\begin{equation}\label{PDE wrt Cartan struct} 
\begin{split}
& -I_{\alpha 2}+J_{\alpha 1}=\mathcal K-1\\
& \ I_{\alpha 3}=-J\\
& \ J_{\alpha 3}=I,
\end{split}
\end{equation}
where the subscripts represent directional derivatives with respect to the coframe $\alpha$, i.e.
$df=f_{\alpha 1}\alpha^{1}+f_{\alpha 2}\alpha^{2}+f_{\alpha 3}\alpha^{3}$, for any function $f$ on $\Sigma$.
\end{proposition}

The proof is a simple computation taking into account the Ricci identities of the coframe 
$(\alpha^1,\alpha^2,\alpha^3)$.

We point out that when regarding the relations \eqref{PDE wrt Cartan struct} as a directional PDE, with the unknown functions $I$, $J$, with respect to the $\mathcal K$-Cartan structure $\alpha$, then the involutivity of this PDE system can be studied by means of Cartan-K\"ahler theorem (see for example \cite{IL2003}). Indeed, Cartan-K\"ahler theory shows that the PDE system \eqref{PDE wrt Cartan struct} is involutive with solutions depending on 2 functions of 2 variables, therefore on the 3-manifold $\Sigma$ there are $(I,J,1)$-generalized Finsler structures depending on 2 functions of 2 variables in the sense of Cartan-K\"ahler theorem as pointed out  in \cite{Br1996}. However this is a quite rough estimation including local, global, as well as trivial solutions.

\begin{remark}
Let $(M,F)$ be a classical compact Finsler surface and let us denote its indicatrix bundle by $\Sigma$. It follows that $\Sigma$ is a compact 3-manifold that admits a natural $(I,J,K)$-generalized Finsler structure induced as follows. Let us denote by $\omega=(\omega^{1},\omega^{2},\omega^{3})$ the $g$-orthonormal coframe on $\Sigma$ induced by the Finslerian structure $F$ (see \cite{BCS2000} for details), where $g$ is the Hessian matrix defined in \eqref{Hessian matrix}. Then $\omega$ satisfies structure equations of type \eqref{finsler_struct_eq} for some functions $I$, $J$ and $K$ on $\Sigma$. These are called the Cartan scalar, the Landsberg curvature and the flag curvature of the Finsler structure $(M,F)$, respectively, being in the same time the invariants of the Finsler structure $F$ in the sense of Cartan's equivalence problem. The functions $I$, $J$ and $K$ are uniquely determined by the fundamental function $F$ only. 

If $(M,F)$ is a Finsler surface of positive constant flag curvature, then the naturally induced generalized Finsler structure on the indicatrix bundle total space $\Sigma$ is an $(I,J,1)$-structure that satisfies all the conditions in Theorem \ref{thm: GFS equivalence to Finsler}. In this case, 
the associated $\mathcal K$-Cartan structure $(\alpha^1,\alpha^2, \alpha^3)$ has some supplementary properties, namely, the leaves of the foliation $\mathcal F=\{\alpha^2=0,\alpha^3+\varphi=0\}$
satisfy the conditions in Theorem \ref{thm: GFS equivalence to Finsler}, where $\varphi:=I\alpha^1+J\alpha^2$ is a 1-form on $\Sigma$ that satisfies the structure equation
\begin{equation}\label{phi_struct_eq}
d\varphi=(\mathcal K-1)\alpha^1\wedge\alpha^2.
\end{equation}
The integral curves of the exterior system $\{\alpha^2=0,\alpha^3+\varphi=0\}$ are called $\varphi$-geodesics in \cite{Br2002}.

Conversely, the construction above can be used to associate an $(I,J,1)$-generalized Finsler structure to any given $\mathcal K$-Cartan structure $(\alpha^1,\alpha^2,\alpha^3)$. Remark that the indicatrix foliation of this $(I,J,1)$-generalized Finsler structure coincides with the foliation 
$\mathcal F$ given above. In the case when the leaves of this foliation satisfy the conditions in Theorem \ref{thm: GFS equivalence to Finsler} we obtain a classical Finsler structure.
\end{remark}
\begin{remark}
We also remark that the left or right invariant $(I,J,1)$-generalized Finsler structures on a Lie group $G$ must have the invariants $I$ and $J$ constant functions on $G$ and implicitly on $G/\Gamma$,  i.e. we obtain only trivial cases that do not interest us. 

\end{remark}

\section{Cartan structures induced $(I,J,1)$-Generalized Finsler structures}\label{sec:5}

We are going to give here a general construction of $(I,J,1)$-generalized Finsler structures induced by $\mathcal K$-Cartan structures.

Indeed, let us consider a $\mathcal K$-Cartan structure $(\alpha^1,\alpha^2,\alpha^3)$, on a closed 3-manifold $\Sigma$, with $\mathcal K$ not necessarily constant. Then from Proposition \ref{prop4.3} it follows that the coframe
\begin{equation}\label{coframe induced IJ1GFS}
\begin{split}
\omega^1&=\alpha^1\\
\omega^2&=\varphi-\alpha^3\\
\omega^3&=\alpha^2
\end{split}
\end{equation}
is an $(I,J,1)$-generalized Finsler structure if and only if the 1-form $\varphi=I\alpha^1+J\alpha^2$ on $\Sigma$ satisfies the structure equation \eqref{phi_struct_eq}.

Therefore, in order to assure the existence of a nontrivial $(I,J,1)$-generalized Finsler structure on 
one of the quotient manifolds $\Sigma=G\slash\Gamma$ given in Theorem \ref{thm: classification}
it suffices to find a 1-form $\varphi$ on $G$ that satisfies the condition:

\begin{condition}\label{C1}
\begin{tabular}{ll}
\  & \ \\
\ & \ 1. $\varphi$ is $\Gamma$-invariant, and\\
\ & \ 2. $\varphi$ satisfies the structure equation \eqref{phi_struct_eq} with non-constant coefficients.
\end{tabular}
\end{condition}

In this case, we have
\begin{proposition}\label{constr 1}
Let $G\in\{SU(2),\widetilde{E}_{2},\widetilde{SL}_{2}\}$ be a Lie group, $\Gamma\subset G$ a discrete subgroup of $G$, $\alpha=(\alpha^{1},\alpha^{2},\alpha^{3})$ a $\mathcal K$-Cartan structure on $G\slash\Gamma$ and $\varphi$ a 1-form on $G$  that satisfies the condition C \ref{C1}. Then, the coframe 
\eqref{coframe induced IJ1GFS} is an $(I,J,1)$-generalized Finsler structure on 
$\Sigma=G\slash \Gamma$.
\end{proposition}

We will compute explicitly $\varphi$ in some special cases in Section \ref{sec:4} showing that finding $\varphi$ and verifying condition C \ref{C1} is far from being trivial. 

In this section we give a rather simple general theoretical construction of such $(I,J,1)$-generalized Finsler structures that will provide a proof for our main result Theorem  \ref{thm: GFS classif}.

Let  $(\alpha^1,\alpha^2,\alpha^3)$ be a $\mathcal K$-Cartan structure and let us consider the conformal $\mathcal{\widetilde K}$-Cartan structure 
$(\widetilde \alpha^1,\widetilde\alpha^2,\widetilde\alpha^3)$ given by
\begin{equation}
\begin{split}
\widetilde \alpha^1& =v\alpha^1\\
\widetilde  \alpha^2& =v\alpha^2\\
\widetilde  \alpha^3& =\alpha^3-*d\log v,
\end{split}
\end{equation}
where $\ast$ denotes the Hodge star operator and $*d\log v=-\dfrac{1}{v}v_{\alpha 2}\alpha^1+\dfrac{1}{v}v_{\alpha 1}\alpha^2$, $v>0$.

Notice that in fact only the pairs of 1-forms $(\alpha^{1},\alpha^{2})$ and 
$(\widetilde\alpha^{1},\widetilde\alpha^{2})$, which define the corresponding Cartan structures, are conformal, i.e. $(\widetilde\alpha^{1},\widetilde\alpha^{2})=v(\alpha^{1},\alpha^{2})$. The expression of 
$\widetilde\alpha^{3}$ follows from Lemma \ref{Lemma 3.3}.

If we take into account that relation \eqref{phi_struct_eq} is equivalent to 
\begin{equation}\label{phi's structure eqn 2}
d\varphi=d\alpha^3-\alpha^1\wedge\alpha^2,
\end{equation}
then by putting 
\begin{equation}
\widetilde\varphi:=*d\log v-\varphi
\end{equation}
it follows
\begin{equation}
d\widetilde\varphi=d*d\log v-d\varphi=
d(\alpha^3-\widetilde  \alpha^3)-d\varphi=-d\varphi+d\alpha^3-\mathcal{\widetilde K}\widetilde\alpha^1\wedge\widetilde \alpha^2
\end{equation}
and taking into account of \eqref{phi's structure eqn 2} we get
\begin{equation}
d\widetilde\varphi=\alpha^1\wedge\alpha^2-\mathcal{\widetilde K}\widetilde\alpha^1\wedge\widetilde \alpha^2,
\end{equation}
or, equivalently
\begin{equation}\label{alg eq for v}
d\widetilde\varphi+\mathcal{\widetilde K}\widetilde\alpha^1\wedge\widetilde \alpha^2
=\dfrac{1}{v^2}(\widetilde \alpha^1\wedge\widetilde \alpha^2).
\end{equation}

This relation still looks like a differential equation, but we would like to regard it as an algebraic equation in $v$. In order to do this we need an 1-form $\widetilde\varphi$ whose differential 
$d\widetilde\varphi$ is spanned only by 
$\widetilde\alpha^1\wedge\widetilde \alpha^2$ and that satisfies

\begin{condition}\label{C2}
\begin{tabular}{ll}
\  & \ \\
\ & \ 1. $\widetilde\varphi$ is $\Gamma$-invariant, and\\
\ & \  2. $d\widetilde\varphi+\mathcal{\widetilde K}\widetilde\alpha^1\wedge\widetilde \alpha^2>0$.
\end{tabular}
\end{condition}

From Proposition \ref{constr 1} we obtain
\begin{proposition}\label{constr 2}
Let $G\in\{SU(2),\widetilde{E}_{2},\widetilde{SL}_{2}\}$ be a Lie group, $\Gamma\subset G$ a discrete subgroup of $G$, $\widetilde\alpha$ a $\widetilde{\mathcal K}$-Cartan structure on 
$G\slash\Gamma$ and $\widetilde\varphi$ a 1-form on $G$  that satisfies the condition C \ref{C2}. Then, the coframe 
\begin{equation}
\begin{split}
& \omega^1=\frac{1}{v}\widetilde\alpha^1\\
& \omega^2=-\widetilde\varphi-\widetilde\alpha^3\\
& \omega^3=\frac{1}{v}\widetilde\alpha^2
\end{split}
\end{equation}
is an $(I,J,1)$-generalized Finsler structure on 
$\Sigma=G\slash \Gamma$, where $v$ is the scalar function obtained by solving the algebraic equation \eqref{alg eq for v}.
\end{proposition}

The simplest way to do this is to consider $(I,J,1)$-generalized Finsler structures induced by 
Liouville-Cartan structures. Let us briefly recall the construction of Liouville-Cartan structures on the quotient manifolds listed in Theorem \ref{thm: classification} (see \cite{GG1995} for details).

Assume that $(\Lambda_0,g)$ is a closed, oriented Riemannian surface (the non-oriented case can be also treated in a slightly different manner) and denote by $Isom_o(\Lambda_0,g)$ its full group of orientation-preserving isometries. If $\mathcal F\subset Isom_o(\Lambda_0,g)$ is a finite group of orientation-preserving isometries of $\Lambda_0$, then $d\mathcal F$ is a finite group of isometries of $ST\Lambda_0$, where it acts freely. Here, $d\mathcal F$ is the set of differentials of the elements of $\mathcal F$, and $ST\Lambda_0$ the unit sphere bundle of $(\Lambda_0,g)$. Hence, the quotient manifold $ST\Lambda_0\slash d\mathcal F$ is a canonical Seifert fibration 
$ST\Lambda_0\slash d\mathcal F\to \Lambda$
over the 2-dimensional orbifold $\Lambda:=\Lambda_0\slash \mathcal F$.

A compact 3-manifold $\Sigma$ can be obtained from this Seifert fibration by simply taking the cover map
\begin{equation}
\Sigma\to ST\Lambda_0\slash d\mathcal F,
\end{equation}
where $\Lambda_0$ is a surface of genus 0, 1 or greater than 1, if and only if $\Sigma$ is a left quotient of $\Sph^3$, $\widetilde E_2$ or $\widetilde{SL}_2$, respectively. Indeed, if $\widetilde \Lambda$ is one of the 2-dimensional Riemannian space form models $\Sph^2$, $\R^2$ or $\mathcal H^2$, respectively, and if $\Sigma=G\slash\Gamma$ for some discrete, cocompact subgroup $\Gamma$ of $G$, then there is a canonical projection $G\to Isom_o(\widetilde \Lambda)$ that maps $\Gamma$ onto its image 
$\Gamma'$. 

If the image $\Gamma'$ is discrete, then $G\slash\Gamma$ has a canonical Seifert fibration 
$G\slash\Gamma\to \Lambda$ over the 2-dimensional orbifold $\Lambda:=\widetilde\Lambda\slash \Gamma$. 

\begin{remark}
It is known that for $\Sph^3$ and $\widetilde{SL}_2$, the image of any discrete subgroup $\Gamma$ under the canonical projection described above is a discrete isometry subgroup $\Gamma'$ on $\Sph^{2}$ and $\mathbb H^{2}$, respectively, but this is not true anymore for $\widetilde E_2$. However, we can work only with cocompact discrete subgroups $\Gamma$ of $G$ whose image $\Gamma'$ is always discrete. Such discrete subgroups $\Gamma$ will be called {\it admissible}. 
\end{remark}

It is remarkable that there always is another description of $\Lambda$, namely $\Lambda=\Lambda_0\slash \mathcal F$, where $(\Lambda_0,g_0)$ is a closed, orientable 2-manifold with some constant curvature metric $g_0$, and $\mathcal F\subset Isom_o(\Lambda_0,g_0)$. 

We also point out that for any admissible subgroup $\Gamma$ of $G$ there exists at least one pair $(\Lambda_0,\mathcal F)$ such that $\Sigma=G\slash\Gamma\to ST\Lambda_0\slash d\mathcal F$ is a covering map. Then, the tautological forms $\{\alpha^1,\alpha^2\}$ of $(\Lambda_0,g)$ determine a Liouville-Cartan structure on $G\slash\Gamma$ that depends on $\Gamma$ only. In other words, different pairs $(\Lambda_0,\mathcal F)$ can yield the same Cartan structure, as the orbifold $\Lambda$ can have different descriptions of the type $\Lambda_0\slash \mathcal F$. 

Therefore, let us consider a closed (oriented) Riemannian surface $(\Lambda_0,g)$ (not necessarily of constant sectional curvature $\mathcal K$) and denote $(\eta^1,\eta^2)$ the $g$-orthonormal coframe on $\Lambda_0$. The structure equations on $\Lambda_0$ are 
\begin{equation}\label{struct eq downstairs}
d\eta^1=a\eta^1\wedge\eta^2,\quad d\eta^2=b\eta^1\wedge\eta^2,
\end{equation}
where $a$, $b$ are the structure functions on $\Lambda_0$. It follows that the Levi-Civita connection form is 
$\eta^3=-a\eta^1-b\eta^2$ and the sectional curvature $\mathcal K=a_{\eta 2} -a^2 -b_{\eta 1} -b^2$, where the subscripts are directional derivatives with respect to the coframe $(\eta^1,\eta^2)$. 

We are led by these arguments to the following {\it construction method}.

\bigskip

We consider {\it any} 1-form $\phi:=f\eta^1+g\eta^2$ on the Riemannian surface 
$(\Lambda_0,g)$ that satisfies the condition
\begin{condition}\label{C3}
\begin{tabular}{ll}
\  & \ \\
\ & \ 1. $\phi$ is $\mathcal F$-invariant, and\\
\ & \  2. $-f_{\eta 2}+af+bg+g_{\eta 1}+\mathcal K>0$,\\
& where $\mathcal F$   is a discrete subgroup of $Isom_0(\Lambda_0,g)$.  
\end{tabular}
\end{condition}

The relation 
\begin{equation}\label{alg_rel_for v downstairs}
d\phi+\mathcal K\eta^1\wedge\eta^2=\frac{1}{v^2}\eta^1\wedge\eta^2
\end{equation}
corresponds to \eqref{alg eq for v} downstairs on $\Lambda_0$.

Remarking that $d\phi=(-f_{\eta 2}+af+bg+g_{\eta 1})\eta^1\wedge\eta^2$, we obtain the formula 2 in the condition C \ref{C3} above
and therefore 
\begin{equation}\label{v general}
v=\dfrac{1}{\sqrt{-f_{\eta 2}+af+bg+g_{\eta 1}+\mathcal K}}.
\end{equation}

All geometrical objects like functions $a$, $b$, $f$, $g$, etc., $\mathcal K$, $v$ and 1-forms
$\eta^1$, $\eta^2$, $\phi$ defined on $\Lambda_0$ naturally lift by $\nu^*$ to $ST\Lambda_0$, where $\nu:ST\Lambda_0\to\Lambda_0$ is the unit sphere bundle of $(\Lambda_0,g)$ 
(we denote $\varphi$, $\widetilde\varphi$ forms \lq\lq upstairs" on $\Sigma$ and by $\phi$ forms \lq\lq downstairs" on $\Lambda_0$).

Then, the coframe
\begin{equation}\label{induced GFS} 
\begin{split}
& \omega^1=\frac{1}{v}\alpha^1\\
& \omega^2=-\nu^*\phi-\alpha^3\\
& \omega^3=\frac{1}{v}\alpha^2
\end{split}
\end{equation}
is an $(I,J,1)$-generalized Finsler structure on $ST\Lambda_0\slash d\mathcal F$ that can be pulled back to $\Sigma$ by the covering map $\Sigma\to ST\Lambda_0\slash d\mathcal F$, where $v$ is given by \eqref{v general},
$\nu:\mathcal F(\Lambda_0)\to \Lambda_0$ is the orthonormal frame bundle of $(\Lambda_0,g)$, which is locally diffeomorphic with $\Sigma$ (we denote here the projection $\nu$ by the same letter with the projection of the unit sphere bundle $ST\Lambda_0\to \Lambda_0$ from obvious reasons), $(\alpha^1,\alpha^2,\alpha^3)$ is the Liouville-Cartan structure induced from $(\Lambda_0, g)$ and $\mathcal F$ on $ST\Lambda_0\slash d\mathcal F$. Taking now into account that 
$v(*d\log v-\phi)=I\eta^1+J\eta^2$ and \eqref{alg_rel_for v downstairs}, the invariants defined on $\Lambda_0$ are 
\begin{equation}\label{IJ generali}
I=-fv-v_{\eta 2},\quad J=-gv+v_{\eta 1},
\end{equation}
where $v$ is again given by \eqref{v general}. In order to keep notations simple we denote $\nu^*I$ and $\nu^*J$ with same letters $I$ and $J$, respectively.

Therefore we obtain

\begin{proposition}\label{construction prop}
Let 
$(\Lambda_0,g)$ be a Riemannian surface of genus 0, 1 or greater than 1,  $\mathcal F\subset Isom_o(\Lambda_0)$ be a finite generated orientation preserving isometry subgroup of $(\Lambda_0,g)$, and let $\Sigma$ be the covering space of $ST\Lambda_0\slash d\mathcal F$ which is diffeomorphic to $G\slash\Gamma$, where $G$ is one of $\Sph^3$, $\widetilde E_2$ or $\widetilde{SL}_2$, respectively.
 
 Then, the Liouville-Cartan structure of $ST\Lambda_0$ induces an $(I,J,1)$-generalized Finsler structure on $\Sigma$ defined by \eqref{induced GFS}, where $\phi$ is a 1-form on $\Lambda_{0}$ that satisfies condition C \ref{C3}. The structure functions $I$, $J$ are given in  \eqref{IJ generali}, where $v$ is the function in \eqref{v general}.
\end{proposition}

\begin{remark}\begin{enumerate}
\item
If needed, one can choose a standard Cartan structure for $(\alpha^1,\alpha^2,\alpha^3)$ obtained as the Liouville-Cartan structure from the space form $\Lambda_0$ and construct the corresponding $(I,J,1)$-generalized Finsler structure as in the Proposition \ref{construction prop}.
\item 
As explained already, due to the $\mathcal F$-invariance condition, the 1-form $\phi$, functions $v$, $f$, $g$, etc. used in the formulas above live in fact on the set $\Lambda=\Lambda_{0}\slash \mathcal F$
which is an orbifold in general. However, in order to avoid complications, we construct examples in Sections 6.2, 6.3 only in the cases when $\Lambda$ is in fact manifold. 
\end{enumerate}
\end{remark}

\begin{proof}[Proof of Theorem \ref{thm: GFS classif}]
Assume that $\Sigma$ admits an $(I,J,1)$-generalized Finsler structure 
$(\omega^1,\omega^2,\omega^3)$. Then $\alpha^1:=\omega^1$, $\alpha^2:=\omega^3$ is a Cartan structure, hence it is a taut contact circle on $\Sigma$ and therefore from Theorem \ref{thm: classification}
the conclusion follows (see Remark \ref{rem: from GFS to Cartan}). 

Conversely, if $\Sigma$ is one of the quotient manifolds given in hypothesis, then from  Theorem \ref{thm: classification} results that $\Sigma$ must carry a Cartan structure (for instance the one induced from Liouville-Cartan structures on a Riemannian surface $(\Lambda, g)$) which we denote by $(\alpha^1,\alpha^2,\alpha^3)$ with the structure function $\mathcal{K}$. If we take any 1-form 
$\varphi:=\nu^*(\phi)$, for any $\phi$ on $\Lambda_0$ that satisfies condition \ref {C3}, 
we obtain that the coframe \eqref{induced GFS} is a non-trivial
$(I,J,1)$-generalized Finsler structure on $\Sigma$ with the structure functions \eqref{IJ generali}. 

$\qedd$
\end{proof}

\subsection{A classical Finsler structure}\label{sec:5.1}
Using a similar argument as in \cite{Br1996}, we use this construction to obtain an $(I,J,1)$-generalized Finsler structure on 
$\Sigma$ that induces a classical Finsler structure on a surface $M$ by the double fibration
\setlength{\unitlength}{1cm} 
\begin{center}
\begin{picture}(7, 3.5)
\linethickness{0.075mm} 
\put(2.5,2.5){$\Sigma\simeq \Sph^{3}$}
\put(2.9,2.4){\vector(1,-1){1.5}}
\put(4.5,0.5){$M\simeq\Sph^{2}$}
\put(4.6,0){$\|$}
\put(3.5,-0.5){$\Sigma\slash\{\omega^{1}=0,\omega^{2}=0\}$}
\put(2.4,2.4){\vector(-1,-1){1.5}}
\put(0.5,0.5){$\Lambda_{0}\simeq\Sph^{2}$}
\put(0.5,0){$\|$}
\put(-0.7,-0.5){$\Sigma\slash\{\omega^{1}=0,\omega^{3}=0\}$}
\put(1.1,1.7){$\lambda$}
\put(4.1,1.7){$\pi$}
\end{picture}
\bigskip
\end{center}

\begin{proposition}
Let us consider $(\Lambda_0=\Sph^2,g)$, where $g$ is a Zoll metric with sectional curvature ${\mathcal K>0}$, and let 
$( \alpha^1, \alpha^2\, \alpha^3)$ be the Liouville-Cartan structure induced by it. Then the coframing
\begin{equation}\label{Zoll induced GFS}
\begin{split}
\omega^1&=\sqrt{\mathcal K}\alpha^1\\
\omega^2&=-\alpha^3\\
\omega^3&=\sqrt{\mathcal K}\alpha^2
\end{split}
\end{equation}
is an $(I,J,1)$-generalized Finsler structure on $\Sigma$. Moreover, this structure projects to a classical Finsler structure on $M:=\Sph^2$.
\end{proposition}
\begin{proof}
Recall that a Zoll metric on $\Sph^2$ is a Riemannian metric all of whose geodesics are closed and have same length (see for example \cite{Be} for details). Equivalently, in this case the manifold of geodesics $M:=\Sph^3\slash\{\alpha^1=0,\alpha^3=0\}=\Sph^2$
is a smooth manifold. If we denote as before the induced Liouville-Cartan structure by 
$(\alpha^1,\alpha^2,\alpha^3)$, then applying the construction in Proposition \ref{construction prop} with $\varphi=0$, it follows 
\begin{equation}
\mathcal K\alpha^1\wedge\alpha^2=\frac{1}{v^2}\alpha^1\wedge\alpha^2,
\end{equation}
i.e. $v=\dfrac{1}{\sqrt{\mathcal K}}$, and this makes sense since for a Zoll metric $\mathcal K>0$. It follows that the coframing \eqref{Zoll induced GFS} is indeed an $(I,J,1)$-generalized Finsler structure on $\Sph^3$ and therefore on $\Sph^3\slash\Gamma$, where $\Gamma$ is the admissible subgroup of $\Sph^3$ corresponding to the cyclic isometry group $\Gamma'=\mathcal C_2$ of $\Sph^2$, which acts freely on 
$\Sph^2$.

The invariants $I$ and $J$ are obtained immediately from 
\begin{equation*}
I\eta^1+J\eta^2=\frac{1}{\sqrt{\mathcal K}}*d \Bigl(\frac{1}{\sqrt{\mathcal K}}\Bigr),
\end{equation*}
namely
\begin{equation*}
I=\frac{1}{2\mathcal K\sqrt{\mathcal K}}\mathcal K_{\eta 2},\qquad J=-\frac{1}{2\mathcal K\sqrt{\mathcal K}}\mathcal K_{\eta 1}.
\end{equation*}

Moreover, remark that the geodesic foliation $\{\alpha^1=0,\alpha^3=0\}$ of $(\Lambda_{0}=\Sph^2,g)$ coincides with the indicatrix foliation $\{\omega^1=0,\omega^2=0\}$ of the $(I,J,1)$-generalized Finsler structure and therefore the conditions of Theorem \ref{thm: GFS equivalence to Finsler} are satisfied. 

$\qedd$
\end{proof}

\section{Concrete constructions}\label{sec:4}

In this section we will construct explicit nontrivial $(I,J,1)$-generalized Finsler structures on each of the quotient manifolds in Theorem \ref{thm: GFS classif}.
\subsection{The case of $SU(2)$}\label{sec:6.1}
Following \cite{GG2002} we identify $SU(2)$ with the unit quaternions $\Sph^3\subset \mathbb H$ by the group isomorphism
\begin{equation*}
g=\begin{pmatrix}
a_0+ia_1 & b_0+ib_1\\
-b_0+ib_1 & a_0-ia_1
\end{pmatrix}\mapsto
a_0+ia_1+jb_0+kb_1,
\end{equation*}
where $i^2=j^2=k^2=ijk=-1$, and $\det g=a_0^2+a_1^2+b_0^2+b_1^2=1$. We consider $a_1$, $b_0$, $b_1$ as local coordinates and always substitute 
\begin{equation}\label{a_0}
a_0=\sqrt{1-a_1^2-b_0^2-b_1^2}.
\end{equation}
This coordinate chart covers only the domain in the hemisphere with $a_0>0$, but this is typical for any non-trivial manifold, i.e. a manifold that is not diffeomorphic to $\R^{n}$. A complete atlas of local charts can be constructed if needed.

We recall that the group operation is given by the Hamiltonian product
\begin{equation}\label{Hamiltonian product}
\begin{split}
g\cdot \bar g & = (a_0+ia_1+jb_0+kb_1)\cdot (\bar a_0+i\bar a_1+j\bar b_0+k\bar b_1)\\
& = (a_0\bar a_0-a_1\bar a_1-b_0\bar b_0-b_1\bar b_1)
+i(a_0\bar a_1+a_1\bar a_0+b_0\bar b_1-b_1\bar b_0)\\
& +j(a_0\bar b_0-a_1\bar b_1+b_0\bar a_0+b_1\bar a_1)
+k(a_0\bar b_1+a_1\bar b_0-b_0\bar a_1+b_1\bar a_0),
\end{split}
\end{equation}
the neutral element is
\begin{equation}
e=\begin{pmatrix}
1 &0\\
0 & 1
\end{pmatrix}\mapsto
1+i0+j0+k0=1
\end{equation}
and inverse element of $g$ is given by
\begin{equation}
g^{-1}=a_0-ia_1-jb_0-kb_1.
\end{equation}

We are going to compute the Maurer-Cartan forms of $SU(2)$ by the usual formula
$\alpha_R=dg\cdot g^{-1}\in \mathfrak{su}(2)$. We have
\begin{equation}
\begin{split}
dg \cdot g^{-1}
& =(da_0+ida_1+jdb_0+kdb_1)\cdot (a_0-ia_1-jb_0-kb_1)\\
&= (a_0da_0+a_1da_1+b_0db_0+b_1db_1)+i(-a_1da_0+a_0da_1-b_1db_0+b_0db_1)\\
&+ j(-b_0da_0+b_1da_1+a_0db_0-a_1db_1)+k(-b_1da_0-b_0da_1+a_1db_0+a_0db_1).
\end{split}
\end{equation}

Taking into account the equation $\det g=1$ we have obtained three right invariant 1-forms 
$(\alpha^1,\alpha^2,\alpha^3)$ on $SU(2)$ considered as subspace of $\mathbb H$, where
\begin{equation}\label{Maurer-Cartan SU(2)}
\begin{split}
\alpha^1=& -b_0da_0+b_1da_1+a_0db_0-a_1db_1\\
\alpha^2=& -b_1da_0-b_0da_1+a_1db_0+a_0db_1\\
\alpha^3=& -a_1da_0+a_0da_1-b_1db_0+b_0db_1. 
\end{split}
\end{equation}
A straightforward computation shows that $(\alpha^1,\alpha^2,\alpha^3)$ satisfy the structure equations \eqref{str_eq_for_stand_Cartan} with $\varepsilon=1$, where we always consider $a_0$ as in \eqref{a_0}.

 To be precise, if we denote $\iota:\Sph^3\to \mathbb H$ the canonical inclusion given by \eqref{a_0}, the standard Cartan structure is $(\iota^*\alpha^1,\iota^*\alpha^2,\iota^*\alpha^3)$, but we will make the difference between $\alpha^i$ and $\iota^*\alpha^i$, $i=1,2,3$ only when necessarily. 

\begin{remark}
We recall that quaternions are in fact pairs of complex numbers obtained by applying the Cayley-Dickson construction to the complex numbers. Indeed, one can represent a vector in $\bbC^2$ as
\begin{equation}
(a_0+ia_1)1+(b_0+ib_1)j=(a_0+ia_1,b_0+ib_1)=(z_1,z_2),
\end{equation}
where $z_1=a_0+ia_1$, $z_2=b_0+ib_1$.

Keeping this identification in mind, we point out that the standard Cartan structure of $SU(2)$ defined in \cite{GG1995} is
\begin{equation}
\begin{split}
\alpha^1+i\alpha^2& =2\iota^*(z_1dz_2-z_2dz_1)\\
&= 2\iota^*(-b_0da_0+b_1da_1+a_0db_0-a_1db_1)+
2i\iota^*(-b_1da_0-b_0da_1+a_1db_0+a_0db_1)
\end{split}
\end{equation}
fact that shows that our Maurer-Cartan forms defined in formulas \eqref{Maurer-Cartan SU(2)} coincide with the standard Cartan forms used in \cite{GG1995}. This is the reason we prefer to work with right invariant 1-forms.
\end{remark}

From the construction of $(I,J,1)$-generalized Finsler structures induced by Cartan structures, we consider the generalized Finsler structure induced by the standard Cartan structure 
$(\alpha^1,\alpha^2,\alpha^3)$ of $SU(2)$ given in \eqref{coframe in I, J}.

 We have seen that for $\varphi=I\alpha^1+J\alpha^2$ it follows from \eqref{phi_struct_eq} 
 that this is a closed 1-form on $SU(2)$ and hence an exact one since $H_1(\Sph^3)=0$. In other words, there exists a function $f:\Sph^3\to \R$ such that $\varphi=df$. If we denote as usual $df=f_{\alpha 1}\alpha^1+f_{\alpha 2}\alpha^2+f_{\alpha 3}\alpha^3$, then we obtain
 \begin{equation}\label{IJformulas}
 I=f_{\alpha 1},\quad J=f_{\alpha 2}, \quad f_{\alpha 3}=0.
 \end{equation}
This means that if we find a function $f(a_1,b_0,b_1)$ on $\Sph^3$ solution of the directional differential equation $f_{\alpha 3}=0$, then directional derivatives of $f$ with respect to $\alpha^1$ and $\alpha^2$ will give $I$ and $J$, respectively. Indeed, one can easily see that taking into account the Ricci identities for $f$ with respect to the $(+1)$ standard Cartan structure 
$(\alpha^1,\alpha^2,\alpha^3)$, namely
\begin{equation}
\begin{split}
& f_{\alpha 21}-f_{\alpha 12}=-f_{\alpha 3}\\
& f_{\alpha 32}-f_{\alpha 23}=-f_{\alpha 1}\\
& f_{\alpha 31}-f_{\alpha 13}=f_{\alpha 2}
\end{split}
\end{equation}
and simply using $f_{\alpha 3}=0$ the conditions \eqref{PDE wrt Cartan struct} are satisfied by $I$ and $J$ in \eqref{IJformulas}.

From \eqref{Maurer-Cartan SU(2)} we have
\begin{equation}
\begin{split}
f_{\alpha 1}= &\dfrac{1}{2}(f_{a1}b_1-f_{b1}a_1+f_{b0}a_0)\\
f_{\alpha 2}= &\dfrac{1}{2}(f_{b0}a_1-f_{a1}b_0+f_{b1}a_0)\\
f_{\alpha 3}= &\dfrac{1}{2}(f_{a1}a_0-f_{b0}b_1+f_{b1}b_0),
\end{split}
\end{equation}
where the subscripts $a1$ etc. of $f$ mean partial derivatives of $f$ with respect to the respective coordinate, and, as usual, we use \eqref{a_0}.

The equation $f_{\alpha 3}=0$ is equivalent to the PDE
\begin{equation}
\dfrac{\partial f}{\partial a_1}\sqrt{1-a_1^2-b_0^2-b_1^2}- 
\dfrac{\partial f}{\partial b_0}b_1
+\dfrac{\partial f}{\partial b_1}b_0=0.
\end{equation}

The general solution of this equation is
\begin{equation}
f(a_1,b_0,b_1)=\Phi\Bigl(b_0^2+b_1^2,\arctan \dfrac{a_0b_0+a_1b_1}{a_0b_1-a_1b_0}\Bigr),
\end{equation}
where $\Phi:\R^{+}\times\R\to \R$ is an arbitrary function of two variables and again we use \eqref{a_0}.

We are next interested in constructing an $(I,J,1)$-generalized Finsler structure on the quotient $SU(2)\slash \Gamma$, where $\Gamma$ is a discrete, therefore finite, subgroup of $SU(2)$. We have
\begin{lemma}\label{lem:4.6}
Let $u,v:SU(2)\to \R$ be any smooth functions, and let $\Gamma$ be a finite subgroup of 
$SU(2)$. If $u$ and $v$ are $\Gamma$-invariant, then
\begin{enumerate}
\item $f(g):=\Phi(u(g),v(g))$ is $\Gamma$-invariant, for any $g\in SU(2)$;
\item  $I:=f_{\alpha 1},\ J:=f_{\alpha 2}$ are also $\Gamma$-invariant,
\end{enumerate}
where the subscripts are the directional derivatives with respect to the coframe \eqref{Maurer-Cartan SU(2)}. 
\end{lemma}
\begin{proof}
If  $u$ and $v$ are $\Gamma$-invariant, i.e. $u(g)=u(g\cdot x)$ and $v(g)=v(g\cdot x)$ for any $g\in SU(2)$ and $x\in \Gamma$, then obviously 
\begin{equation*}
f(g\cdot x)=\Phi(u(g\cdot x),v(g\cdot x))=\Phi(u(g),v(g))=f(g)
\end{equation*}
 and (1) is proved.

Under the hypothesis conditions, since $f$ is $\Gamma$-invariant from (1), it follows that 
$\varphi=df$ is also   $\Gamma$-invariant, i.e. $I\alpha^1+J\alpha^2$ is $\Gamma$-invariant and taking into account that $(\alpha^1$, $\alpha^2)$ are Maurer-Cartan forms on $SU(2)$, the conclusion (2) follows.

$\qedd$
\end{proof}
In other words, in oder to get a $\Gamma$-invariant $(I,J,1)$-generalized Finsler structure, it is enough to verify if $u$ and $v$ are $\Gamma$-invariant.

In our case $u(g)=b_0^2+b_1^2$ for any $g=a_0+ia_1+jb_0+kb_1$, and from \eqref{Hamiltonian product} it follows
\begin{equation*}
\begin{split}
u(g\cdot x)& =(a_0^2+a_1^2)(y_0^2+y_1^2)+(b_0^2+b_1^2)(x_0^2+x_1^2)\\
&+2(a_0b_0+a_1b_1)(x_0y_0-x_1y_1)+2(a_0b_1-a_1b_0)(x_0y_1+x_1y_0),
\end{split}
\end{equation*}
where $x=x_0+ix_1+jy_0+ky_1\in \Gamma$.
Putting now the condition $u(g\cdot x)=u(g)$ one can easily see that the only admissible element 
$x\in\Gamma$ must have the form $x=x_0+ix_1$, with $x_0^2+x_1^2=1$. 

Similarly, since $v(g)=\arctan\dfrac{a_0b_0+a_1b_1}{a_0b_1-a_1b_0}$, we are going to verify the $\Gamma$-invariance of $w(g):=\dfrac{a_0b_0+a_1b_1}{a_0b_1-a_1b_0}=\dfrac{\mathcal A(g)}{\mathcal B(g)}$. A straightforward computation gives 
$w(g\cdot x)=\dfrac{\mathcal A(g\cdot x)}{\mathcal B(g\cdot x)}$, where
\begin{equation}
\begin{split}
\mathcal A(g\cdot x)& =(a_0b_0+a_1b_1)(x_0^2-x_1^2-y_0^2+y_1^2)+(a_0^2+a_1^2-b_0^2-b_1^2)(x_0y_0+x_1y_1)\\
& +2(a_0b_1-a_1b_0)(x_0x_1-y_0y_1)\\
\mathcal B(g\cdot x)& =(a_0b_1-a_1b_0)(x_0^2-x_1^2+y_0^2-y_1^2)+(a_0^2+a_1^2-b_0^2-b_1^2)(-x_1y_0+x_0y_1)\\
& -2(a_0b_0+a_1b_1)(x_0x_1+y_0y_1).
\end{split}
\end{equation}
Putting the condition $\dfrac{\mathcal A(g\cdot x)}{\mathcal B(g\cdot x)}=
\dfrac{\mathcal A(g)}{\mathcal B(g)}$ it follows
\begin{eqnarray}
& & x_0y_0+x_1y_1=0\label{1}\\
& & x_0x_1-y_0y_1=0\label{2}\\
& & x_1y_0-x_0y_1=0\label{3}\\
& & x_0x_1+y_0y_1=0\label{4}
\end{eqnarray}

From \eqref{2} and \eqref{4} we get $x_0x_1=0$ and $y_0y_1=0$. If we assume, for example, $x_0=0$, then from \eqref{1} and \eqref{3} we get $x_1y_1=0$ and $x_1y_0=0$. If, moreover we take $x_1=0$ then we get $y_0y_1=0$ and from here $y_1=0$ and thus $x_1y_0=0$ or $y_0=0$ and thus  
$x_1y_1=0$. In conclusion, in the case $x_0=0$, two coordinates among $x_1,y_1,y_0$ must vanish. A  similar analysis can be done taking in turn $x_1=0$, $y_0=0$ and $y_1=0$, respectively. We conclude that in order to assure the $\Gamma$-invariance of $w$ it is necessary that three among the coordinates $x_0,x_1,y_1,y_0$ must vanish.

Putting all these together we can formulate
\begin{lemma}\label{lem:4.7}
With the notation above we have
\begin{enumerate}
\item $u(g\cdot x)=u(x)$, for $x=x_0+ix_1\in SU(2)$ and any $g\in SU(2)$,
\item $v(g\cdot x)=v(x)$, for $x\in \{x_0, ix_1, jy_0, ky_1\}\subset SU(2)$ and any $g\in SU(2)$.
\end{enumerate}
\end{lemma}

We are going to see if there are some particular finite subgroups $\Gamma$ of $SU(2)$ such that $f$ is $\Gamma$-invariant. We recall that if $\Gamma$ is a finite subgroup of $SU(2)$, then it must be one of the following: the cyclic group $\mathcal C_m$ of order $m$, the binary dihedral group $\mathcal D_{4n}^*$ of order $4n$, the binary tetrahedral group $\mathcal T^*$ of order $24$, the binary octahedral group $\mathcal O^*$ of order $24$, the binary icosahedral group $\mathcal I^*$ of order $120$ (see for eg. \cite{W} p. 87--88).

\begin{proposition}
Let $(\alpha^1,\alpha^2,\alpha^3)$ be the Maurer-Cartan forms of $SU(2)$. Then, the coframing 
\begin{equation}\label{coframe with df}
\begin{split}
& \omega^1=\alpha^1\\
& \omega^2=df-\alpha^3\\
& \omega^3=\alpha^2
\end{split}
\end{equation}
is an $(I,J,1)$-generalized Finsler structure, with invariants $I=f_{\alpha 1}$, $J=f_{\alpha 2}$, that descends on $SU(2)\slash\mathcal C_m$, where 
$f:SU(2)\to \R$ is a smooth function such that
\begin{itemize}
\item $f(a_1,b_0,b_1):=
\Phi(b_0^2+b_1^2, 
\arctan \frac{a_0b_0+a_1b_1}{a_0b_1-a_1b_0})$, for $m=2$, and
\item $f(a_1,b_0,b_1):=\Phi(b_0^2+b_1^2)$, for $m>2$,
\end{itemize}
where $\Phi$ is an arbitrary function of two or one variables, respectively. 
\end{proposition}

\begin{proof}
Since $\mathcal C_m$ is the cyclic group of order $m$, it is generated by $x=\cos\frac{2\pi}{m}+i\sin\frac{2\pi}{m}$. From Lemma \ref{lem:4.7} one can see that both $u$ and $v$ are $\mathcal C_2$-invariant, while only $u$ is $\mathcal C_m$-invariant for $m>2$. In the second case the assumption that the arbitrary function $\Phi$ depends only on $u$ suffices to obtain a $\mathcal C_m$-invariant function $f$. The rest follows from Lemma \ref{lem:4.6} and the discussion above.

$\qedd$
\end{proof}

\begin{proposition}
Let $(\alpha^1,\alpha^2,\alpha^3)$ be the Maurer-Cartan forms of $SU(2)$. Then, the coframing \eqref{coframe with df}
is an $(I,J,1)$-generalized Finsler structure which descends on $SU(2)\slash \mathcal D_8^*$, with the invariants $I=f_{\alpha 1}$, $J=f_{\alpha 2}$, where 
$f:SU(2)\to \R$ is a smooth function such that $f(a_1,b_0,b_1):=\Phi(\arctan \dfrac{a_0b_0+a_1b_1
}{a_0b_1-a_1b_0}
)$, 
where $\Phi$ is an arbitrary even function of one variable.
\end{proposition}

\begin{proof}
The binary dihedral group  $\mathcal D_{4n}^*=\{x,y:x^2=(xy)^2=y^n\}$ of order $4n$ has the generators
$x=i$ and $y=\cos\frac{\pi}{n}+j\sin\frac{\pi}{n}$. A straightforward computation followed an application of Lemma  \ref{lem:4.7} shows that $v$ is $\mathcal D_{8}^*$-invariant, while $u$ is not $\mathcal D_{4n}^*$-invariant for any $n$, nor $v$ for $n>2$. Here $\mathcal D_{8}^*=\{\pm 1, \pm i, \pm j, \pm k\}$ is the quaternion group denoted also by $\mathcal Q_8$. The rest follows remarking that $v(g\cdot x)=\pm v(g)$ for $x\in  \mathcal D_{8}^*$ and any $g\in SU(2)$ and following the proof above.

$\qedd$
\end{proof}
\begin{remark}
One can see that since the binary tetrahedral group $\mathcal T^*=\mathcal Q_8^{x,y}\rtimes \mathcal C_3^z$ is generated by $x=i$, $y=j$ and $z=-(1+i+j+k)$, nor $u$ either $v$ can be invariant. The same is true for $\mathcal O^*$ and $\mathcal I^*$. Therefore we cannot construct by this method $(I,J,1)$-generalized Finsler structures on $SU(2)/\Gamma$, for $\Gamma\in\{\mathcal T^*, \mathcal O^*, \mathcal I^*\}$. Nevertheless, the existence of other $(I,J,1)$-generalized Finsler structures is not eliminated.
\end{remark}


\subsection{The case of $\widetilde{E}_2$}\label{sec:4.2}

Let us recall that the group of Euclidean motions of the plane $ASO(2)$ is made of transformations of the form 
\begin{equation}
\begin{pmatrix}
X\\Y
\end{pmatrix}
\mapsto 
\begin{pmatrix}
x\\y
\end{pmatrix}
+R
\begin{pmatrix}
X\\Y
\end{pmatrix},
\end{equation}
where $R\in SO(2)$ is a rotation matrix. It is customary to write it as a matrix Lie group containing matrices of the form
\begin{equation}
\begin{pmatrix}
1 & 0\\
Z & R
\end{pmatrix},\qquad Z=(x,y)^t\in \E^{2},\quad 
R=\begin{pmatrix}
\cos\theta & -\sin\theta\\
\sin\theta & \cos\theta
\end{pmatrix}
\in SO(2),
\end{equation}
or, simply
\begin{equation}
ASO(2)=\{(x,y,\theta): (x,y)^t\in \E^{2},\ \theta\in[0,2\pi)\}
\end{equation}
(see for example \cite{IL2003} for details on $ASO(2)$).

The universal covering $\widetilde E_2$ of $ASO(2)$ is obtained by allowing any real value for $\theta$, i.e. $\widetilde E_2$ is the subgroup of $\mathbb R^3$ with the multiplication
\begin{equation}
(x_1,y_1,\theta_1)\cdot(x_2,y_2,\theta_2)=\Bigg(
\begin{pmatrix}
\cos\theta_1 & -\sin\theta_1\\
\sin\theta_1 & \cos\theta_1
\end{pmatrix}
\begin{pmatrix}
x_2\\y_2
\end{pmatrix}
+\begin{pmatrix}
x_1\\y_1
\end{pmatrix},
\theta_1+\theta_2\Bigg).
\end{equation}

It can be seen therefore that the standard metric in $\mathbb R^3$ leads to left-invariant metrics on $\widetilde E_2$ under this identification (see \cite{GG1995} for details).

We can now realize as $\widetilde E_2=\{(x,y,\theta):z=x+iy,w=\lambda+i\theta\in \mathbb C^2\}\subset \mathbb C^2$, where $(z=x+iy,w=\lambda+i\theta)$ are the standard coordinates of $\mathbb C^2$, and we denote the inclusion $\iota:\widetilde E_2\to \mathbb C^2$. 

It can be now easily checked that the real and imaginary parts of the 1-form
\begin{equation}\label{global MC on E2}
\alpha^1+i\alpha^2=\iota^*(e^{-w}dz)
\end{equation}
give the standard Cartan structure on $\widetilde E_2$. The remaining 1-form $\alpha^3$ being
\begin{equation}
\alpha^3=\iota^*(idw).
\end{equation}

In the local coordinates of $\widetilde E_2$ we get
\begin{equation}
\begin{split}
& \alpha^1=\cos \theta dx+\sin\theta dy\\
& \alpha^2=-\sin\theta dx+\cos\theta dy\\
& \alpha^3=-d\theta.
\end{split}
\end{equation}

Indeed, it can be easily seen that these 1-forms satisfy \eqref{global MC on E2}.

\begin{remark}
Using these coordinates, we can see that the PDE system \eqref{PDE wrt Cartan struct} reads
\begin{equation}
\begin{split}
& I_\theta=J\\
& J_\theta=-I\\
& \sin\theta(I_x+J_y)+\cos\theta(J_x-I_y)=1
\end{split}
\end{equation}
and that it has the solution 
\begin{equation}
 \begin{pmatrix}
I \\ J
\end{pmatrix}=
\begin{pmatrix}
\cos\theta & \sin\theta\\
-\sin\theta & \cos\theta
\end{pmatrix}
\begin{pmatrix}
\eta(x,y)\\
f(x,y)
\end{pmatrix}
\end{equation} 
where we put $\eta(x,y)=\int  f_xdy+y-g(x)$
for any arbitrary functions $f=f(x,y)$ and $g=g(x)$ defined in the plane $(x,y)$.

In order to complete the construction, we need to obtain explicit functions $I$ and $J$ invariant to some admissible subgroup $\Gamma$ of $\widetilde E_2$. Although this is possible, the computations involved are cumbersome. 
\end{remark}

Instead of working directly on the 3-manifold $\Sigma$ we prefer to do the construction using Liouville-Cartan structures taking into account that all $\mathcal K$-Cartan structures come from Liouville-Cartan structures on Riemannian surfaces and that all are conformal to the standard one. 

Following the same strategy as in Proposition \ref{construction prop}
 we consider the Riemannian space form 
$(\Lambda_0,g)=(\mathbb E^2,can)$, where $can$ is the canonical metric on $\E^{2}$, and recall that its full group of isometries is $\R\times O(2)$. 

An orientation preserving subgroup $\mathcal F$ of isometries of the plane, acting freely and discretely on $\mathbb E^2$,  is generated by translations only (see \cite{S} p. 407). This is a discrete subgroup of $\R^2$, the group of all translations of $(\mathbb E^2,can)$. It follows that $\mathcal F$ is isomorphic to $\mathbb Z$ or 
$\mathbb Z\oplus \mathbb Z$ being generated by one or two translations, respectively. 

In this case, $\Lambda:=\Lambda_0\slash\mathcal F$ is a smooth surface, namely an open cylinder or a torus, if $\mathcal F$ is generated by a translation or two translations, respectively. We will restrict here to the compact quotient case, i.e. the 2-dimensional torus $\Lambda=T^{2}$ case.

What we need now is any 1-form $\phi=f(x,y)dx+g(x,y)dy$ on $\Lambda_0$ which is  
$\mathcal F$-invariant, where $can=dx^2+dy^2$. Since the coframe  $(\eta^{1},\eta^{2})=(dx,dy)$ is $\mathcal F$-invariant by definition we need two  $\mathcal F$-invariant functions of two variables defined in plane. If we denote $\mathcal F=<\tau_1,\tau_2>$, where $\tau_1:(x,y)\mapsto (x+1,y)$ and 
$ \tau_2:(x,y)\mapsto (x,y+1)$, then we are looking for functions $f$, $g$ periodic of period 1 in plane satisfying the supplementary condition $-f_y+g_x>0$. 

One way to find such functions is to take any two functions $f$ and $g$ of two variables, or equivalently any 1-form $\phi$,  directly on the torus $T^2$ such that $-f_y+g_x>0$.

Indeed, let us remark that the torus $T^{2}$ is an orientable surface, therefore it exists a nowhere zero 2-form $\Theta$ on $T^{2}$ that gives the orientation, i.e. we can choose $\Theta>0$ or $\Theta<0$ everywhere. 

On the other hand, since $H^{2}(T^{2})=0$ it follows that it always exists a 1-form $\phi=f(x,y)dx+g(x,y)dy$ on $T^{2}$ such that $d\phi=\Theta$, where $(x,y)$ are some local coordinates on 
$T^{2}$. 

Therefore, for the convenient orientation on $T^{2}$, we always have a 1-form $\phi$ on $T^{2}$ such that $d\phi>0$ everywhere. 


In this case, condition 2 in \ref{C3} reads $d\phi=\frac{1}{v^{2}}dx\wedge dy$ and hence 
$v=\frac{1}{\sqrt{-f_{y}+g_{x}}}$, where $f_{y}$ and $g_{x}$ represent usual partial derivatives of the functions $f$ and $g$ with respect to the variables $y$ and $x$, respectively. 

Having all these done, \eqref{induced GFS} implies that the coframe 
\begin{equation}\label{GFS coframe on E_{2}}
\begin{split}
& \omega^1= \sqrt{-f_y+g_x}\cdot\alpha^1\\
& \omega^2= -\nu^*(\phi)-\alpha^3\\
& \omega^3= \sqrt{-f_y+g_x}\cdot \alpha^2
\end{split}
\end{equation}
is an $(I,J,1)$-generalized Finsler structure on $\Sigma=\widetilde{E}_2\slash \Gamma$, one of the five $T^2$-bundles over $\Sph^1$ with periodic monodromy, where $(\alpha^{1},\alpha^{2}, \alpha^{3})$ is the standard Cartan structure on $\widetilde E_{2}$, $\nu:ST(T^{2})\to T^{2}$ is the usual canonical bundle projection,
 and $\phi$ is the 1-form on $T^{2}$ defined above. 
 The subgroup $\Gamma$ is the inverse image of $\mathcal F$ through the projection map $\widetilde{E}_2 \mapsto Isom_o(\mathbb E^2, can)$. The invariants of this $(I,J,1)$-generalized Finsler structure are
 \begin{equation}
 \begin{split}
 & I=-\frac{f}{\sqrt{-f_{y}+g_{x}}}-\frac{f_{yy}-g_{xy}}{2(-f_{y}+g_{x})\sqrt{-f_{y}+g_{x}}},\\
 & J=-\frac{g}{\sqrt{-f_{y}+g_{x}}}+\frac{f_{xy}-g_{xx}}{2(-f_{y}+g_{x})\sqrt{-f_{y}+g_{x}}}.
 \end{split}
 \end{equation}

In the case of $\mathcal F$ isomorphic to $\mathbb Z\oplus \mathbb Z$, i.e. generated by two translations, the corresponding discrete group $\Gamma$ of $\widetilde E_2$ is the discrete group with monodromy matrix 
$A_1=\begin{pmatrix}
1 & 0 \\
0 & 1 
\end{pmatrix}
$
and the 3-manifold $\Sigma$ is diffeomorphic with the 3-torus $T^3$. We point out that starting with Liouville-Cartan structures, we obtain only homothety classes of Cartan structures with all leaves of the codimension one  foliation $\{\alpha^1=0,\alpha^2=0\}$ closed (see \cite{GG1995}, Theorem 7.4).

Using complex coordinates, the 2-dimensional real torus $\Lambda_{0}=T^{2}$ can be identified with the 1-dimensional complex torus $\Lambda_{0}=\mathbb C\slash \Gamma'$, where $\Gamma'=<1,z_{0},z_{1}>$. Here $z_{0}\in \mathcal H^{2}$ and $z_{1}\in \mathbb C$ are arbitrary. Remark that $\Gamma'$ is discrete provided $qz_{1}\in<1,z_{0}>$ for some positive integer $q$ (one can take $q$ minimal, namely the order of $z_{1}$ modulo $<1,z_{0}>$).

Then the corresponding discrete subgroup $\Gamma$ of $\widetilde E_{2}\subset \mathbb C^{2}$ is the lattice 
\begin{equation}\label{lattice Gamma on E_{2}}
\Gamma=<(1,0),(z_{0},0), (z_{1},2\pi i r_{1})>,
\end{equation}
where $z_{0},z_{1}$ are as above and $r_{1}$ is a positive integer. It follows that the 3-manifold $\Sigma=\widetilde E_{2}\slash\Gamma$ is the 3-torus $T^{3}=\Lambda_{0}\times \Sph^{1}$ which covers $ST\Lambda_{0}$ (see \cite{GG1995}, p. 206 for details). 

Therefore the coframe \eqref{GFS coframe on E_{2}} is an $(I,J,1)$-generalized Finsler structure on the compact 3-manifold $\Sigma=\widetilde E_{2}\slash\Gamma=T^{3}$, where $\Gamma$ is given in 
\eqref{lattice Gamma on E_{2}}.

Similarly, one can obtain $\Gamma'$-invariant Liouville-Cartan structures on $\E^{2}$, where $\Gamma'$ is another discrete subgroup of $Isom_{o}(\E^{2},can)$ and compute its corresponding discrete subgroup $\Gamma$ of $\widetilde E_{2}$, but the construction we gave above is enough for proving the following result.

\begin{proposition}\label{lemma: E_2}
The standard $0$-Cartan structure on $\Sigma=\widetilde{E}_2\slash \Gamma=T^{3}$ induces a non-trivial  $(I,J,1)$-generalized Finsler structure on the 3-manifold $\Sigma$, where $\Gamma$ is the lattice  \eqref{lattice Gamma on E_{2}}.
\end{proposition}


\subsection{The case of $\widetilde{SL}_2$}\label{sec:4.3}


Let us denote by $\mathcal H^2$ the upper half plane in $\mathbb C$, the coordinates of 
$\mathcal H^2\times \mathbb C$ with $(z,\widetilde w)$ and the corresponding point in 
$\mathcal H^2\times \mathbb C^*$ with $(z,e^{\widetilde w})=(z,w)$.

We identify $\widetilde{SL}_2$, the universal covering of the unit tangent bundle $ST\mathcal H^2$, with 
$\mathcal H^2\times i\R$ with coordinates $(z,\theta)$, where we denote by $z=x+iy\in \mathbb C$ the standard complex coordinate, and $\widetilde w=\lambda +i\theta$. We consider the inclusion 
$\iota: \mathcal H^2\times (i\R)\to \mathcal H^2\times \mathbb C$, i.e. the 3-manifold  
$\widetilde{SL}_2\equiv  \mathcal H^2\times (i\R)$ is transversal to the vector $\dfrac{\partial}{\partial \lambda}$.

It can be easily verified that the real and imaginary parts of the 1-form
\begin{equation}\label{E_{2} standard Cartan forms}
\alpha^1+i\alpha^2=\iota^*(e^wdz)
\end{equation}
give the standard Cartan structure on $\widetilde{SL}_2$. 

Since the natural inclusion $\iota:\widetilde{SL}_{2}\to\bbC^{2}$ is given by 
$\iota(z,i\theta)=(z,0+i\theta)$ we have
\begin{equation*}
\iota^{*}dw=-\iota^{*}d\bar w=id\theta.
\end{equation*}

On the other hand, by adding and then subtracting formula \eqref{E_{2} standard Cartan forms} with its complex conjugate $\alpha^{1}-i\alpha^{2}$ we obtain concrete formulas for $\alpha^{1}$ and $\alpha^{2}$. Using now Lemma \ref{Lemma 3.3} it follows that the remaining 1-form $\alpha^3$ is
\begin{equation}
\alpha^3=-i\iota^*(dw).
\end{equation}

If we consider the 3-manifold $\widetilde{SL}_2$ to be defined by the equations $(x=x,y=y,\lambda=\log y,\theta=\theta)$, as a submanifold of $\mathcal H^{2}\times \mathbb C$, then in local coordinates $(x,y,\theta)$ of $\widetilde{SL}_2$ we obtain
\begin{equation}
\begin{split}
& \alpha^1=\frac{1}{y}(\cos \theta dx+\sin\theta dy)\\
& \alpha^2=\frac{1}{y}(-\sin\theta dx+\cos\theta dy)\\
& \alpha^3=-d\theta-\frac{1}{y}dx.
\end{split}
\end{equation}


\begin{remark}
Using these coordinates, we can see that the PDE system \eqref{PDE wrt Cartan struct}  reads
\begin{equation}
\begin{split}
& I_\theta=-J\\
& J_\theta=I\\
& \sin\theta(yI_x+I_\theta+xJ_y)-\cos\theta(yJ_x+J_\theta-xI_y)=2
\end{split}
\end{equation}
and it has the solution 
\begin{equation}
\begin{split}
& I(x,y,\theta)=\frac{2x-C}{y}\sin\theta\\
& J(x,y,\theta)=\frac{2x-C}{y}\cos\theta
\end{split}
\end{equation} 
for any arbitrary constant $C$.

Similarly with the case of $\widetilde E_2$, although it is possible to check the invariance of these functions under the action of the discrete subgroups of $\widetilde{SL}_2$, since all Cartan structures are conformal to the standard one and this corresponds to the Liouville-Cartan structure induced from the space form $(\mathcal H^2,can)$, we prefer to follow the construction in Section \ref{sec:5}.

\end{remark}

Let us recall that the orientation preserving isometries group $Isom_o(\mathcal H^2,can)$ can be identified with the group $PSL_2$, which is the quotient of $SL_2$ by its central subgroup of order two (see \cite{S}), namely $PSL_{2}=SL_{2}\slash\{\pm I_{2}\}$, where $I_{2}$ is the 2x2 identity matrix. 

Indeed, the action of $SL_2$ on $\mathcal H^2$ is defined by the usual {\it M\"obius transformation}
\begin{equation}
A
\cdot z=\dfrac{az+b}{cz+d},
\end{equation}
where $z=u+iv$, $v>0$, $A=\begin{pmatrix}
a & b \\
c & d
\end{pmatrix}\in\mathcal M_{2\times 2}(\R)$, 
and $\det(A)=1$, and a simple computation shows that $A\cdot z=(-1)A\cdot z$. It follows that the equivalence relation $A\sim B$ if and only $A=\pm B$, defined on $SL_{2}$, is natural and we obtain the quotient space $PSL_{2}:=SL_{2}\slash\sim$, that motivates the definition above.

\begin{remark}
It is also known that $PSL_{2}$ can be identified with the unit tangent bundle $ST\mathcal H^{2}$ by choosing any vector $\xi\in ST\mathcal H^{2}$ at some fixed point $(x,y)\in \mathcal H^{2}$ and defining the map
\begin{equation*}
PSL_{2}\to ST\mathcal H^{2},\qquad A\mapsto A_{*}(\xi),
\end{equation*}
where $A_{*}$ is the differential of $A$ (see \cite{GG1995}, Subsection 5.3 for details).

It can be seen that under this map, the natural metric on $ST\mathcal H^{2}$ induced from 
$(\mathcal H^{2},can)$ is pulled back to a left invariant metric on $PSL_{2}$, which is independent of the choice of $\xi$. This metric can now be pulled back to $\widetilde{SL}_{2}$ 
by the universal covering mapping $\widetilde{SL}_{2}\to PSL_{2}$, 
if necessary. 

The universal covering $\widetilde{SL}_{2}$ of ${SL}_{2}$ is an example of finite dimensional Lie group that is not a matrix group, i.e. $\widetilde{SL}_{2}$ admits no faithful, finite-dimensional representation.

This means that the mapping $G=\widetilde{SL}_{2}\to Isom_{o}\Lambda_{0}=PSL_{2}$ that maps the discrete subgroup $\Gamma$ into its image $\Gamma'$ is just the universal covering map.
\end{remark}

We also recall that the elements $A\in PSL_{2}$ can be classified by their traces as follows. If the trace of the matrix $A$ is less, equal or greater than 2, the element $A$ of $ PSL_{2}$ is called {\it elliptic}, {\it parabolic} or {\it hyperbolic}, respectively. 

It is also known that a subgroup $\mathcal F$ of $PSL_{2}$ is discrete if and only if it acts properly discontinuously on $\mathcal H^{2}$, namely the following three conditions hold
\begin{enumerate}
\item the $\mathcal F$-orbit of any point is locally finite;
\item  the $\mathcal F$-orbit of any point is discrete and the stabilizer of that point is finite;
\item for any point, there is a neighborhood of that point, say $V$, for which only finitely many 
$A\in \mathcal F$ satisfy $A(V)\cap V\neq \emptyset$.
\end{enumerate}

Such a discrete subgroup $\mathcal F$ is called a {\it Fuchsian group}.

\begin{remark}\label{Fuchs groups propert}
\begin{enumerate}
\item It is known that a Fuchsian group is abelian if and only if it is cyclic.
\item Any hyperbolic or parabolic cyclic subgroup of $PSL_{2}$ is Fuchsian, while an elliptic cyclic subgroup is Fuchsian if and only if it is finite.
\end{enumerate}
\end{remark}

It can be seen that $PSL_2$ is generated by the matrices 
\begin{equation}\label{Ai matrices}
A_1:=\begin{pmatrix}
1 & s \\
0 & 1
\end{pmatrix}
,\quad 
A_2:=\begin{pmatrix}
\lambda & 0 \\
0 & \frac{1}{\lambda}
\end{pmatrix},\quad 
A_3:=\begin{pmatrix}
0 & 1 \\
-1 & 0
\end{pmatrix},
\end{equation}
which are parabolic, hyperbolic and elliptic elements of $PSL_{2}$, respectively, where $s\in \mathbb R$ and $\lambda>0$.

These matrices correspond to  three types of isometries: {\it translations, dilatations} and {\it rotations}, respectively, which form the full group of orientation preserving isometries of $(\mathcal H^2, can)$. 

Let us point out here that another kind of isometry of $(\mathcal H^{2}, can)$ is the {\it reflection} $R$, but one can see that this is not an orientation preserving matrix, therefore unlikely $A_{i}$, $i=1,2,3$, it results $R\notin PSL_{2}$. 
By direct computation taking into account the following actions 
\begin{equation}
A_1:(x,y)\mapsto (x+s, y),\quad A_2:(x,y)\mapsto \lambda^2(x,y),
\quad A_{3}:(x,y)\mapsto (-\dfrac{x}{x^2+y^2},\dfrac{y}{x^2+y^2})
.
\end{equation}

In order to obtain the corresponding discrete orientation preserving subgroups $\mathcal F$ of 
$(\mathcal H^{2}, can)$, let us remark that, since $A_{1}$ and $A_{2}$ are parabolic and hyperbolic, respectively, for some fixed $s$ and $\lambda>0$, the cyclic groups $\mathcal F_{1}(s):=<A_{1}(s)>$, $\mathcal F_{2}(\lambda):=<A_{2}(\lambda)>$ are Fuchsian groups due to Remark \ref{Fuchs groups propert}. Here we denote by $A_{1}(s)$ and $A_{2}(\lambda)$ the matrices $A_{1}$ and $A_{2}$ in \eqref{Ai matrices} defined for some fixed $s$ and $\lambda>0$, respectively.

Similarly, in the case of $A_{3}$, let us observe that $(A_{3})^{2}=-I_{2}$ and therefore the corresponding discrete subgroup $<A_{3}>$ is any finite group of order 2, hence $\mathcal F_{3}:=<A_{3}>$ is also Fuchsian. 

It is known that the hyperbolic 2-orbifold $\Lambda=\Lambda_{0}\slash\mathcal F$ is a manifold when $\mathcal F$ contains no elements of finite order, in other words, when $\mathcal F$ contains no elliptic elements.

Therefore, for some fixed $s$ and $\lambda>0$, the quotient sets $\Lambda_{1}=\mathcal H^{2}\slash \mathcal F_{1}(s)$ and $\Lambda_{2}=\mathcal H^{2}\slash \mathcal F_{2}(\lambda)$ are hyperbolic 2-manifolds. Taking into account the discussion above, it follows $\Lambda_1$ can be identified with the open upper half cylinder, and $\Lambda_2$ with an open transversally sectionned torus (a cylinder cut by 2 transversal plans without both lids).

What we need now is to find two functions of two variables $f(x,y)$, $g(x,y)$ defined on $\mathcal H^2$, which are invariant under the action of a finite order discrete subgroup $\mathcal F\subset Isom_o(\mathcal H^2, can)$ and such that 
\begin{equation}\label{phi cond on SL2}
-f_{\eta 2}+f+g_{\eta_1}-1>0,
\end{equation}
 where subscripts here are directional derivatives with respect to the orthonormal coframe $(\eta^1,\eta^2)$ of $(\mathcal H^2, can)$, namely $df=f_{\eta 1}\eta^1+f_{\eta 2}\eta^2$, and similar for $g$. Indeed, the structure equations \eqref{struct eq downstairs}
 of this coframe are 
\begin{equation*}
d\eta^1=\eta^1\wedge\eta^2,\quad d\eta^2=0, 
\end{equation*}
and therefore the structure functions $a$ and $b$ take the values 1 and 0, respectively.

 In $(x,y)$-coordinates the canonical metric $can$ on $\mathcal H^{2}$ is given by $\dfrac{(dx)^{2}+(dy)^{2}}{y^{2}}$ and therefore we have 
\begin{equation}
\eta^1=\dfrac{1}{y}dx,\quad \eta^2=\dfrac{1}{y}dy.
\end{equation} 
Moreover, the Levi-Civita connection form is $\eta^3=-\dfrac{1}{y}dx$, the Gauss curvature is $\mathcal K=-1$, and hence, from \eqref{phi cond on SL2},
 we obtain the condition 
\begin{equation}\label{sl2 condition}
-yf_y+f+yg_x-1>0,
\end{equation}
where the subscripts represent usual partial derivatives, namely $df=f_xdx+f_ydy$. One can easily see that 
$(f_{\eta 1},f_{\eta 2})=y(f_x,f_y)$. 

\bigskip

In the case of the {\it translations} group $\mathcal F_{1}$, we can consider $f=f(y)$ and $g=g(y)$, functions invariant to the translation subgroup defined above. In this case, \eqref{sl2 condition} reads $-yf'(y)+f-1>0$, and $g(y)$ arbitrary, but this kind of $f$ always exists (for instance $f(y)=my+n$, for any $m>0$ and $n>1$ will do). Therefore, 
for any $\phi=\dfrac{1}{y}(f(y)dx+g(y)dy)$, the coframe 
\begin{equation}\label{translations induced GFS SL2} 
\begin{split}
& \omega^1=\sqrt{-yf'(y)+f-1}\ \alpha^1\\
& \omega^2=-\nu^*\phi-\alpha^3\\
& \omega^3=\sqrt{-yf'(y)+f-1}\ \alpha^2
\end{split}
\end{equation}
 is a $(I,J,1)$-generalized Finsler structure on $\Sigma=\widetilde{SL}_2\slash \Gamma$, where $\Gamma$ is the lift of the translations group $\mathcal F_{1}$ on $(\mathcal H^2, can)$ to $\widetilde{SL}_{2}$. The structure functions are given by
\begin{equation}
I=-\frac{1}{\sqrt{-yf'(y)+f-1}}[f+\frac{y^2f''(y)}{2(-yf'(y)+f-1)}],\quad J=-\frac{1}{\sqrt{-yf'(y)+f-1}}g(y).
\end{equation}

Indeed, since $v^{2}=\frac{1}{-yf'(y)+f-1}$, by direct computation we get $v_x=0$, $v_y=\dfrac{yf''(y)}{2}v^3$ and therefore $v_{\eta 1}=0$, $v_{\eta 2}=yv_y=\dfrac{y^2f''(y)}{2}v^3$ and from \eqref{IJ generali} the formulas above follow. 

In the case of $f$ linear function, we get $I=-\frac{1}{\sqrt{-yf'(y)+f-1}}f$, $J=-\frac{1}{\sqrt{-yf'(y)+f-1}}g$, with $g$ arbitrary function of one variable. Obviously this is non trivial 
$(I,J,1)$-generalized Finsler structure. 

\bigskip

In the case of {\it dilatations} group $\mathcal F_{2}$, if we consider $f(x,y)=\bar f(\tau)_{|\tau=\frac{x}{y}}$ and $g(x,y)=\bar g(\tau)_{|\tau=\frac{x}{y}}$, these functions are invariant to dilatations. A stronger restriction would be $f=0$, hence we have $\phi=\bar g(\tau)_{|\tau=\frac{x}{y}}\eta^{2}$, and $d\phi=\bar g'(\tau)\tau_{\eta 1}\eta^1\wedge\eta^2$. In this case, we get 
$(\tau_{\eta 1},\tau_{\eta 2})=(1,-\tau)$, and
imposing the condition 
$\bar g'(\tau)\tau_{\eta 1}-1=\bar g'(\tau)-1>0$, we get $v^2=\dfrac{1}{\bar g'(\tau)-1}$. A straightforward computation shows 
\begin{equation}
v_x=-\frac{v^{3}}{2y}\bar g''(\tau),\quad v_y=\frac{xv^{3}}{2y^2}\bar g''(\tau),
\end{equation}
and from here
\begin{equation}
v_{\eta 1}=-\frac{v^{3}}{2}\bar g''(\tau),\quad v_{\eta 2}=\frac{xv^{3}}{2y}\bar g''(\tau).
\end{equation}

Therefore, the coframe 	
\begin{equation}\label{dilatations induced GFS SL2} 
\begin{split}
& \omega^1=\sqrt{\bar g'(\tau)-1}\ \alpha^1\\
& \omega^2=-\nu^*\phi-\alpha^3\\
& \omega^3=\sqrt{\bar g'(\tau)-1}\ \alpha^2
\end{split}
\end{equation}
for any $\phi=\dfrac{1}{y}\bar g(\dfrac{x}{y})dy$, $\bar g(\tau)>\tau$ is an $(I,J,1)$-generalized Finsler structure on $\Sigma=\widetilde{SL}_2\slash \Gamma$, where $\Gamma$ is the lift of the finite 
order dilatations group 
$\mathcal F_{2}$ of $(\mathcal H^2, can)$ to the universal covering space $\widetilde{SL}_{2}$.

The structure functions are given by
\begin{equation}
I=-\frac{x}{2y}
\frac{\bar g''(\tau)}{(\bar g'(\tau)-1)^{\frac{3}{2}}},\quad 
J=-\frac{\bar g''(\tau)}{2(\bar g'(\tau)-1)^{\frac{3}{2}}}-
\frac{\bar g(\tau)}{\sqrt{\bar g'(\tau)-1}},\quad 
\tau=\frac{x}{y}.
\end{equation}

A similar study can be done for other examples of Fuchsian groups, but these two cases are enough for our purposes. 

We have
\begin{proposition}\label{lemma: SL_2}
The standard $(-1)$-Cartan structure on $\Sigma=\widetilde{SL}_2\slash \Gamma$ induces a non-trivial  $(I,J,1)$-generalized Finsler structures on the 3-manifold $\Sigma$, where $\Gamma$ is one of the 
lifts of the groups $\mathcal F_{1}$ or $\mathcal F_{2}$ to $\widetilde{SL}_2$.
\end{proposition}


\section{Relation with the conformal classes of Cartan structures}\label{sec:6}

 Recall that the {\bf homothety class}  of a contact circle $(\alpha^1,\alpha^2)$ is the collection of all pairs $(\tilde{\alpha}^1,\tilde{\alpha}^2)$  obtained from $(\alpha^1,\alpha^2)$ by multiplication by positive functions $v$ and rotation by a constant angle $\theta$.

The {\bf conformal class}  of a contact circle $(\alpha^1,\alpha^2)$ is the collection of all pairs $(\tilde{\alpha}^1,\tilde{\alpha}^2)$  obtained from $(\alpha^1,\alpha^2)$ by multiplication by the same positive function $v$.

 If a contact circle is $\mathcal K$-Cartan, then all contact circles homothetic to it are also $\mathcal K$-Cartan provided $dv\in<\alpha^1,\alpha^2>$.

\begin{theorem} [\cite{GG2002}]\label{conf classes of K-Cartan}
\begin{enumerate}
\item  Let $\Sigma=SU(2)\slash \Gamma$ be a (compact) left quotient of $SU(2)$. 
Then on $\Sigma$, up to homothety and diffeomorphism, 
the set of Cartan structures is given by  descending to $\Sigma$ the following family of Cartan structures living on $\Sph^3$:
\begin{equation}\label{sphere family}
\alpha^1+i\alpha^2=4\iota^\ast \left(az_1dz_2-(1-a)z_2dz_1\right), \qquad 0 < a<1,
\end{equation}
where $\iota$ is the standard inclusion of $\Sph^3$ in $\mathbb C^2$. For $a=\frac{1}{2}$ one obtains the standard Cartan structure.
\item Let $\Sigma=\widetilde{SL}_2\slash \Gamma$ be a (compact) left quotient of $\widetilde{SL}_2$. Then in each conformal class of Cartan structures on $\Sigma$ there is one and only one (-1)-Cartan structure.
\item Let $\Sigma=\widetilde{E}_2\slash \Gamma$ be a (compact) left quotient of $\widetilde{E}_2$. Then each conformal class of Cartan structures on $\Sigma$ contains a (0)-Cartan structure. This (0)-Cartan structure is unique up to multiplication by a positive constant.
\end{enumerate}
\end{theorem}

Let us recall that our construction of $(I,J,1)$-generalized Finsler structures given in Proposition \ref{construction prop} works for arbitrary Cartan structures, regardless they are standard or not. 

In other words, for each $\mathcal K$-Cartan structure $(\alpha^1,\alpha^2,\alpha^3)$,  on $\Sigma=G\slash\Gamma$, in any of the conformal classes described in Theorem \ref{conf classes of K-Cartan}, for a 1-form $\varphi$ on $\Sigma$, that satisfies C \ref{C1} or C \ref{C2}, there is a naturally associated $(I,J,1)$-generalized Finsler structure on $\Sigma$. Indeed, if $G$ is $\widetilde E_2$ or $\widetilde{SL}_2$, then roughly speaking all $\mathcal K$-Cartan structures are conformal to the standard Cartan structure on $\Sigma$, and therefore all are in the same time Liouville-Cartan structures arising as in the construction in Section \ref{sec:5}, where $\Gamma$ is any admissible subgroup of $G$. The same conclusion holds for $G=\Sph^3$ and $\Gamma$ non-abelian. 

We recall that if $G=\Sph^3$ and $\Gamma$ abelian, then there are $\mathcal K$-Cartan structures on $G\slash\Gamma$ which are not conformal to the standard Cartan structure. This is the case of any $\{\alpha^1,\alpha^2\}$ in the family \eqref{sphere family} obtained for any $a\in (0,1)$, $a\neq \frac{1}{2}$. These Cartan structures do not arise from Liouville-Cartan structures on Riemannian surfaces. 
However, one can easily construct  $(I,J,1)$-generalized Finsler structures in this case as follows. 

If $\Gamma$ is abelian, then $\Sigma=SU\slash \Gamma$ must be a lens space 
$L(m,m-1)$, $m\in \mathbb{N}$, and its $\mathcal K$-Cartan structures are the family \eqref{sphere family} for any $a\in (0,1)$. 

Let us briefly recall the definition of a lens space.

On the unit sphere $\Sph^3=\left\{(z^1,z^2)\in \mathbb C^2:\left|z^1\right|^2+\left|z^2\right|^2=1\right\}$, the map 
\[
(z^1,z^2)\mapsto\left(e^{i\phi^1}z^1,e^{i\phi^2}z^2\right), \qquad 0\leq \phi^1,\phi^2 < 2\pi
\]
defines an isometric action of the torus $\Sph^1 \times \Sph^1$.
\par
Now, let $p,q \in \mathbb N$ relatively prime and $1\leq p <q$ and denote by $\mathbb Z_q$ the cyclic group of order $q$. Composing the homomorphism 
\[
\mathbb Z_q \rightarrow \Sph^1 \times \Sph^1, \qquad r \mapsto \left(e^{2\pi i r/q},e^{2\pi i r p/q}\right)
\]with the above action, it follows that $\mathbb Z_q$ acts isometrically on $\Sph^3$. This action has no fixed points because $p$ and $q$ are relatively prime and moreover, the quotient $\Sph^3/\mathbb Z_q$ is a manifold, denoted $L(p,q)$, and called \textit{lens space}.

Intuitively speaking, a lens space is a $3$--dimensional manifold obtained by gluing two solid tori together along their boundaries.

On $L(m,m-1)$ all $\mathcal K$-Cartan structures, up to homothety and diffeomorphism, are obtained by descending to $L(m,m-1)$ the family  \eqref{sphere family}  living on $\Sph^3$. 

With these notations it can be seen that $(\alpha^1,\alpha^2,\alpha^3)$ is a $\mathcal K$--Cartan structure with
\begin{equation}\label{alpha_{3}, K}
\begin{split}
& \alpha^3=\iota^\ast\left[\frac{1}{ai}\overline{z}_1dz_1+\frac{1}{(1-a)i}\overline{z}_2dz_2\right]\\
& \mathcal K=\frac{1}{8a(1-a)\iota^\ast \left[az_1\overline{z}_1+(1-a)z_2\overline{z}_2\right]}>0.
\end{split}
\end{equation}

We remark that $\mathcal K$ and $\alpha^{3}$ are real 0- and 1-forms on $\Sph^{3}$ written in complex coordinates, respectively. One can easily see that $\mathcal K= 1$ only if $a=1/2$.

Therefore, starting with any $\mathcal K$-Cartan structure in the family \eqref{sphere family}, we can construct $(I,J,1)$-generalized Finsler structures on the lens space $L(m,m-1)$, $m\in \mathbb{N}$, making use of the technique developed in Section \ref{sec:5.1}. Indeed, the simplest way to do this is to take the coframe
\begin{equation}\label{new coframe}
\begin{split}
&\omega^1=\sqrt{\mathcal K}\ \alpha^1\\
&\omega^2=-\alpha^3\\
&\omega^3=\sqrt{\mathcal K}\ \alpha^2,
\end{split}
\end{equation}
where $\{\alpha^{1},\alpha^{2}\}$ is given by \eqref{sphere family} and $\alpha^{3}$, $\mathcal K$ by \eqref{alpha_{3}, K}.

It can be easily seen that $\{\omega^1,\omega^2,\omega^3\}$ is a non-trivial $(I,J,1)$-generalized Finsler structures with the structure functions
\begin{equation*}
I=\frac{1}{2\mathcal K\sqrt{\mathcal K}}\mathcal K_{\alpha 2},\qquad J=-\frac{1}{2\mathcal K\sqrt{\mathcal K}}\mathcal K_{\alpha 1}.
\end{equation*}
provided $\mathcal K\neq 1$, i.e. $a\neq\frac{1}{2}$.

More sophisticated constructions are also possible by taking a $\Gamma$-invariant 1-form $\phi$ on $\Sigma$, but the details of such constructions are too complicated for the purpose of this paper. 

There is another interesting fact about the $(I,J,1)$-generalized Finsler structures \eqref{new coframe} constructed above.  

If one denotes by $Diff(\Sigma)$ the diffeomorphism group of $\Sigma$, and by $C(\Sigma)$ the space of homothety classes of Cartan structures on $\Sigma$, then the {\it moduli space} of Cartan structures is  given by $\mathcal{M}(\Sigma)=C(\Sigma)\slash Diff(\Sigma)$ (see \cite{GG2002}). One can easily see that there is a 1-to-1  correspondence between $\mathcal K$-Cartan structures and the $(I,J,1)$-generalized Finsler structures \eqref{new coframe} up to diffeomorphism and conformal equivalence. 

Indeed, let us give a definition. An $(I,J,1)$-generalized Finsler structure $\omega=(\omega^{1}, \omega^{2}, \omega^{3})$ on a closed 3-manifold $\Sigma=SU(2)\slash \Gamma$ is called {\it $\mathcal K$-induced} if there exists a $\mathcal K$-Cartan structure $\alpha=(\alpha^{1}, \alpha^{2}, \alpha^{3})$, $\mathcal K>0$, such that \eqref{new coframe} holds good. The structure functions $I$, $J$ are obtained from the relation
\begin{equation}
I\alpha^{1}+J\alpha^{2}=\frac{1}{\sqrt{K}}*d(\frac{1}{\sqrt{K}}).
\end{equation}

We will introduce an equivalence relation of $\mathcal K$-induced $(I,J,1)$-generalized Finsler structures as follows. If $\omega$ and $\tilde{\omega}$ are two such structures, then $\omega\sim\tilde\omega$ if and only if the corresponding $\mathcal K>0$-Cartan structures $\alpha$ and $\tilde\alpha$ are conformal equivalent, i.e. there exists a function $v>0$ on $\Sigma=SU(2)\slash\Gamma$ such that 
$(\tilde\alpha^{1},\tilde{\alpha}^{2})=v(\alpha^{1},\alpha^{2})$. We denote by $C_{GFS}(\Sigma)$ the space of equivalence classes of $\mathcal K$-induced $(I,J,1)$-generalized Finsler structures on $\Sigma$ and consider the moduli space of such structures $\mathcal M_{GFS}(\Sigma)=C_{GFS}(\Sigma)/Diff(\Sigma)$.

Then the mapping
\begin{equation}
\Phi:\mathcal M_{GFS}(\Sigma)\to \mathcal M(\Sigma),\quad [\omega]\mapsto [\alpha]_{conf}
\end{equation}
associates to each $\bar\omega\in[\omega]$ the corresponding conformal $\mathcal K$-Cartan structure $\alpha$ from the conformal class  $[\alpha]_{conf}$ as described above. One can see that this is a 1-to-1 correspondence that does not depend on the choice of representatives of the equivalence classes. It follows that we can identify the moduli spaces $\mathcal M_{GFS}(\Sigma)$ and $\mathcal M(\Sigma)$ and from \cite{GG1995}, Theorem 7.4 we obtain the following

\begin{theorem}
\begin{enumerate}
\item If $G=SU(2)$ and $\Gamma$ is a non-abelian discrete subgroup, then $\mathcal{M}_{GFS}(\Sigma,\omega)$ consists of a single point. 
\item  If  $G=SU(2)$ and $\Gamma$ is abelian, then $\Sigma$ must be a lens space $L(m,m-1)$, $m\in \mathbb{N}$, and the moduli space is 
\begin{equation*}
\mathcal{M}_{GFS}(L(m,m-1),\omega)  = \{a\in \mathbb{C}:0<\rm{Re}(a)<1\}\slash(a\sim 1-a).
\end{equation*}
\end{enumerate}
\end{theorem}

\begin{remark}
We remark that our construction of the induced  $(I,J,1)$-generalized Finsler structures from taut contact circles in Proposition \ref{prop: correspondence} is quite naive, but this is enough for proving Theorem  \ref{thm: GFS classif}. More sophisticated constructions, for example linear combinations of the 1-forms $\alpha$ can be imagined. These can lead to other special generalized Finsler structures. A simple example related to some previous work (\cite{SSS2010}, \cite{SSS2011}) is presented in the Appendix.
\end{remark}

\section{Appendix}
Let us recall that a {\it generalized Landsberg structure} on a 3-manifold $\Sigma$ is an $(I,0,K)$-generalized Finsler structure that satisfies the structure equations \eqref{GFS} for $J=0$.

If $(\Sigma, \omega^1,\omega^2,\omega^3)$ is such a structure, we consider the set of 1-forms $(\alpha^1,\alpha^2)$ given by
\begin{equation}\label{alpha1, alpha2}
\alpha^{1}:=\omega^{2},\quad \alpha^{2}:=m\omega^{3},
\end{equation}
where $m$ is a smooth function on the 3-manifold $\Sigma$.



\begin{proposition}\label{GLS}
 If $(\Sigma,\omega)$ is an $(I,0,K)$-generalized Finsler structure, then the pair $(\alpha^1,\alpha^2)$ is a Cartan structure on $\Sigma$ if and only if the function m satisfies the conditions:
\begin{enumerate}
\item $m$ is nowhere vanishing on $\Sigma$,
\item $m_1=0$,
\item $m^2=\frac{1}{K}$,
\end{enumerate}
where we write $dm=m_1\omega^1+m_2\omega^2+m_3\omega^3$.
\end{proposition}

\begin{proof} The proof is purely computational. 
\end{proof}

Let us remark that the conditions in Proposition \ref{GLS} impose some restrictions on the generalized 
Landsberg structure $(\Sigma, \omega^1,\omega^2,\omega^3)$ as well. Indeed, we must have
\begin{equation*}
K>0\qquad K_1=0
\end{equation*}
everywhere on $\Sigma$. 

It can now be easily seen that an $(I,0,K)$-generalized Finsler structure satisfying the conditions above naturally induces a $\mathcal K$-Cartan structure. We have 

\begin{proposition}
Let $\Sigma$ be a closed 3-manifold and let $(\Sigma,\omega)$ be an $(I,0,K)$-generalized Finsler structure such that $K>0$, $K_{1}=0$. Then the coframe
\begin{equation}
\begin{pmatrix}
\alpha^{1}\\ \alpha^{2}\\ \alpha^{3}
\end{pmatrix}=
\begin{pmatrix}
0 & 1 & 0 \\
0 & 0 & \frac{1}{\sqrt{K}}\\
\sqrt{K} & 0 & -\frac{1}{2K\sqrt{K}}K_{2}
\end{pmatrix}
\begin{pmatrix}
\omega^{1}\\ \omega^{2}\\ \omega^{3}
\end{pmatrix}
\end{equation}
is a $\mathcal K$-Cartan structure with structure constant
\begin{equation}
\mathcal K=K-\frac{3}{4}\frac{1}{K}\Bigl(\frac{1}{K}(K_{2})^{2}-\frac{2}{3}K_{22}   \Bigr),
\end{equation}
where the subscripts are directional derivatives with respect to the coframe $\omega=(\omega^{1},\omega^{2},\omega^{3})$.
\end{proposition}
\begin{proof}
Since $(\alpha^{1},\alpha^{2})$ given by \eqref{alpha1, alpha2} is a Cartan structure, 
Lemma \ref{Lemma 3.3} implies that there exists a unique 1-form $\alpha^{3}$ such that 
\begin{equation}\label{d}
\begin{split}
&d\alpha^1=\alpha^2\wedge\alpha^{3}\\
&d\alpha^2=\alpha^{3}\wedge\alpha^1.
\end{split}
\end{equation}

A straightforward computation shows that 
\begin{equation}
\alpha^{3}=\frac{1}{m}\omega^1-m_2\omega^3.
\end{equation}

One can see that $\alpha^1\wedge\alpha^2\wedge\alpha^{3}=\omega^1\wedge\omega^2\wedge\omega^3\neq 0$, therefore $(\alpha^1,\alpha^2,\alpha^{3})$ is also a coframe on $\Sigma$.

By exterior derivation it follows
\begin{equation}
d\alpha^{3}=\frac{1}{m}\Bigl(\frac{1}{m}-m_{22} \Bigr)\alpha^1\wedge\alpha^2
\end{equation}

and therefore we can conclude that $(\alpha^1,\alpha^2,\alpha^{3})$ is a $\mathcal K$--Cartan structure with 
\begin{equation}
\mathcal K=\frac{1}{m}\Bigl(\frac{1}{m}-m_{22} \Bigr).
\end{equation}
Using now the conditions on $m$ found in Proposition \ref{GLS} the conclusion follows.

$\qedd$
\end{proof}

\begin{remark}
Do not confound the structure function $\mathcal K$ of the Cartan structure $\alpha$ with the structure function $K$ of the generalized Finsler structure $\omega$. 
\end{remark}

Conversely, by similar computations as above, we obtain
\begin{proposition}
Let $(\Sigma,\alpha)$ be a $\mathcal K$-Cartan structure on the closed 3-manifold $\Sigma$ and 
let $m:\Sigma\to \R\setminus\{0\}$ a smooth function that satisfies the relations
\begin{equation}\label{directional PDE in m}
\begin{split}
& m_{\alpha 3}=0\\
& m_{\alpha 11}=\frac{1}{m}-m\mathcal K\\
& m_{\alpha 12}=0,
\end{split}
\end{equation}
where the subscripts represent directional derivatives with respect to the coframe $\alpha=(\alpha^{1},\alpha^{2},\alpha^{3})$.
Then the coframe
\begin{equation}
\begin{pmatrix}
\omega^{1} \\ \omega^{2} \\ \omega^{3}
\end{pmatrix}=
\begin{pmatrix}
0 & m_{\alpha 1} & m\\
1 & 0 & 0 \\
0 & \frac{1}{m} & 0
\end{pmatrix}
\begin{pmatrix}
\alpha^{1} \\ \alpha^{2} \\ \alpha^{3}
\end{pmatrix}
\end{equation}
is an $(I,0,K)$-generalized Finsler structure on $\Sigma$ with the structure functions 
\begin{equation}
I=2m_{\alpha 2},\qquad K=\frac{1}{m^{2}}.
\end{equation}
\end{proposition}
\begin{remark}
The  involutivity of the directional PDE \eqref{directional PDE in m} can be studied by means of Cartan-K\"ahler theory. We have done such a study for a similar directional PDE in \cite{SSS2010} and construct explicit solutions in \cite{SSS2011}.
\end{remark}


Since any Cartan structure is a taut contact circle, from Theorem \ref{thm: classification} we obtain

\begin{theorem}
If $(\Sigma, \omega)$ is an $(I,0,K)$-generalized Finsler structure on a closed 3-manifold $\Sigma$
satisfying the conditions $K>0$, $K_1=0$, then $\Sigma$ is diffeomorphic to $G\slash\Gamma$, where $G$ and 
$\Gamma$ are as in Theorem \ref{thm: classification}.
\end{theorem}
Finally, let us point out that this construction doesn't lead to a classical Landsberg structure by the reasons presented in \cite{SSS2010} and \cite{SSS2011}.


\bigskip

\medskip

\begin{center}

Sorin V. SABAU\\

School of Science, Department of Mathematics\\
Tokai University,\\
Sapporo, 
005\,--\,8601 Japan

\medskip
{\tt sorin@tspirit.tokai-u.jp}

\bigskip

Kazuhiro SHIBUYA\\
Graduate School of Science, Hiroshima University, \\
Higashi Hiroshima, 739\,--\,8521, Japan

\medskip
{\tt shibuya@hiroshima-u.ac.jp}

\bigskip
Gheorghe Piti\c s\\
Department of Mathematics and Informatics\\
University Transilvania of Bra\c sov, \\
Bra\c sov, Romania

\medskip
{\tt gh.pitis@unitbv.ro}

\end{center}

\end{document}